\documentclass[12pt]{article}
\usepackage{mathrsfs}
\usepackage{amsfonts}
\usepackage{enumitem}
\usepackage{amsmath, amsfonts, amssymb, amsthm, amscd}
\usepackage[all]{xy}
\usepackage[noblocks]{authblk}

\theoremstyle{plain}
\newtheorem{thm}{Theorem}[section]
\newtheorem{theorem}[thm]{Theorem}
\newtheorem{lemma}[thm]{Lemma}

\newtheorem{corollary}[thm]{Corollary}
\newtheorem{proposition}[thm]{Proposition}

\newtheorem{definition}[thm]{Definition}

\theoremstyle{remark}

\newtheorem{example}[thm]{Example}

\newtheorem{defn-thm}[thm]{Definition-Theorem}
\renewcommand{\bar}{\overline}

\renewcommand{\phi}{\varphi}

\newcommand{\C}{{\mathbb C}}

\newcommand{\M}{{\mathcal M}}
\newcommand{\T}{{\mathcal T}}

\renewcommand{\tilde}{\widetilde}

\newcommand{\g}{{\mathfrak g}}
\newcommand{\EQ}{{\ \Longleftrightarrow \ }}

\def\Y{X}

\def\P{\Phi}
\def\TT{\mathcal {T}_m^H}

\def\C{\mathbb{C}}

\def\a{\mathfrak a}
\def\d{\partial}

\def\i{\sqrt{-1}}

\def\T{\mathcal{T}}

\def\X{\mathfrak{X}}

\def\M{\mathcal {M}}
\def\Z{\mathcal {Z}_m}
\def\sZ{\mathcal {Z}}

\def\ZZ{\mathcal {Z}_m^H}

\def\U{\mathcal {U}}
\def\V{\mathcal {V}}

\title{From local Torelli to global Torelli}
\author{Kefeng Liu, Yang Shen}

\begin{document}
\date{}
\maketitle

\vspace{-20pt}

\begin{abstract}
We introduce the notions of strong local Torelli and T-class for polarized manifolds,
and prove that strong local Torelli implies global Torelli theorem on the Torelli spaces for polarized manifolds in the T-class.
  We discuss many new examples of projective manifolds for which such global Torelli theorem holds. As applications we prove that,
  in these cases, a canonical completion  of the Torelli space is a bounded pseudoconvex domain in complex Euclidean space, and show that
   the global Torelli theorem holds on moduli space with certain level structure.
\end{abstract}

\parskip=5pt
\baselineskip=15pt





\setcounter{section}{-1}
\section{Introduction}
Let $(X,L)$ be a polarized manifold with $X$ a projective manifold and $L$ an ample line bundle on $X$. Let $\M$ denote the moduli  space of polarized manifolds, or certain
smooth cover of the moduli space of polarized manifolds, which contains $(X,L)$. Consider the period map
$$\Phi_{\M}:\, \M\to D/\Gamma,$$
from the moduli space $\M$ to the period domain $D$, modulo the action of the monodromy group $\Gamma$.

Among the central problems in Hodge theory are the local Torelli problem
and the global Torelli problem. As discussed in \cite{Griffiths4}, the
local Torelli problem is the question of deciding when the Hodge
structure on $\text{H}^{n}(X_{p},\C)$ separates points in any
local neighborhood in $\M$. In case that $\M$ is smooth, the local Torelli problem is
equivalent to that the tangent map of $\Phi_{\M}$ is injective at any point in $\M$, which is also called infinitesimal Torelli problem.


The global Torelli problem is the question about
when the Hodge structure on $\text{H}^{n}(X_{p},\C)$ uniquely characterizes
 the polarized algebraic structure on $X_{p}$, which is
equivalent to the question that when the period map $\Phi_{\M}$ is
globally injective.

The infinitesimal Torelli problem is solved for a large class of compact
complex manifolds, for examples, algebraic curves $X$, if $g(X)=2$
or if $g(X)>2$ and $X$ is not hyper-elliptic, K3-Surfaces,
Calabi-Yau manifolds, hyperk\"ahler manifolds, hypersurfaces and
complete intersections in $\C \text{P}^{n}$ with a few exceptions.
We refer the reader to Chapter VIII in \cite{Grif84} for a detailed survey.
Moreover, many criteria have been found to decide when the infinitesimal Torelli
theorem holds.

The global Torelli problem is solved only for
several special cases, and there seems to be no systematic methods
to study the global Torelli problem. In many cases, only generic Torelli type theorem
 can be proved, which means that the period
map is injective on some open and dense subset of
$\M$. From our work in this paper, one will see that this difficulty is mostly due to the
complicatedness of the moduli space and the monodromy group
$\Gamma$.

To avoid such difficulties in proving the global Torelli theorem, we lift the period map to the universal cover of the moduli space,  which we call the Teichm\"uller space, and
consider the lifted period map. Furthermore, we consider the moduli space of polarized manifolds with level $m$ structure, as well as the Torelli space which is an irreducible
component of the moduli space of marked and polarized manifolds.  We notice that Torelli space is also called the moduli space of framed polarized manifolds in the literature,
for example \cite{Beau}.

More precisely, let $(X,L)$ be a polarized manifold. For simplicity we also denote by $L$ its first Chern
class. We fix a lattice $\Lambda$ with a pairing
$Q_{0}$, where $\Lambda$ is isomorphic to
$H^n(X_{0},\mathbb{Z})/\text{Tor}$ for some polarized manifold $X_{0}$ in $\M$, and
$Q_{0}$ is defined by the cup-product. For a polarized manifold
$(X,L)\in \M$, we define a marking $\gamma$ as an isometry of the lattices
\begin{equation}\label{marking}
\gamma :\, (\Lambda, Q_{0})\to (H^n(X,\mathbb{Z})/\text{Tor},Q).
\end{equation}

Let $m$ be any integer $\geq 3$.
We follow the definition of Szendr\"oi \cite{sz}
to define an $m$-equivalent relation of two markings on $(X,L)$ by
$$\gamma\sim_{m} \gamma' \text{ if and only if } \gamma'\circ \gamma^{-1}-\text{Id}\in m \cdot\text{End}(H^n(X,\mathbb{Z})/\text{Tor}),$$
and denote by $[\gamma]_{m}$  the set of all the $m$-equivalent
classes of $\gamma$. Then we call $[\gamma]_{m}$ a level $m$
structure on the polarized manifold $(X,L)$.

Let $\mathscr{L}_{m}$ be the moduli space of polarized algebraic manifolds with level $m$ structure that contains $(X,L)$. We then introduce the notion of {\em T-class}, where
the letter T stands for Torelli. The purpose of introducing T-class is for us to work on smooth coverings of the moduli spaces.

\begin{definition}\label{intr T-class}
We say that the polarized manifold $(X,L)$ belongs to T-class if there exists an integer $m_{0}\ge 3$ such that the irreducible component $\Z$ of $\mathscr{L}_{m}$ containing $(X,L)$ is a complex manifold, on which there is an analytic family $f_{m}:\,\U_{m}\to \Z$ for any $m\ge m_{0}$.
\end{definition}
Clearly we can assume that $m_{0}=3$ without loss of generality.

Let $\T$ be the universal cover of
$\Z$ with the pull-back family $g_{m}:\, \U\to \T$. It will be proved
that, in our setting, $\T$ is independent of the choice
of $m$. See Lemma \ref{independent of m}. We will call $\T$ the
Teichm\"uller space of the  polarized manifold $(X, L)$.

Let $\T'$ be an irreducible component of the moduli space of marked and polarized manifolds containing $(X, L)$ in the T-class, which we will call the Torelli space in this
paper. We will see that  $\T'$ is a connected complex manifold which is a covering space of  $\Z$ for each $m\geq 3$.  Therefore as the universal cover of $\Z$, the
Teichm\"uller space $\T$ is also a universal cover of $\T'$. See Section \ref{Hodge-Teich} for detailed discussions about moduli and Torelli spaces.

There are many examples
of polarized manifolds that belong to the T-class, such as Calabi-Yau manifolds, hyperk\"ahler
manifolds, many hypersurfaces and complete intersections in
projective spaces.

With the above notations, we can define the period map $$\Phi_{\Z}:\, \Z\to D/\Gamma$$
and the lifted period map $\Phi:\, \T \to D$ such that the diagram
$$\xymatrix{
\T \ar[r]^-{\Phi} \ar[d]^-{\pi_m} & D\ar[d]^-{\pi_{D}}\\
\Z \ar[r]^-{\Phi_{\Z}} & D/\Gamma }$$ is commutative, where
$\pi_{m}:\, \T\to \Z$ is the universal covering map, and $$\pi_D:\, D\to D/\Gamma$$ denotes the natural quotient map.
We can also define the period map $$\Phi':\, \T' \to D$$ from the Torelli space $\T'$, such that the following diagram commutes
$$\xymatrix{
\T \ar[dr]^-{\pi}\ar[rr]^-{\Phi} \ar[dd]^-{\pi_m} && D\ar[dd]^-{\pi_{D}}\\
&\T'\ar[ur]^-{\Phi'}\ar[dl]_-{\pi'_{m}}&\\
\Z \ar[rr]^-{\Phi_{\Z}} && D/\Gamma,
}$$
where the maps $\pi'_{m}:\, \T'\to \Z$ and $\pi:\, \T\to \T'$ are the natural covering maps.

Since $\Z$ is smooth, the local Torelli is equivalent to that the tangent maps of $\Phi_{\Z}$ and $\Phi$ are everywhere nondegenerate.
To proceed further, we introduce the notion of {\em strong local Torelli}. The purpose of introducing this new notion is to construct an affine structure on certain natural completion space of the Torelli space $\T'$. The affine structure is one of the most crucial ingredients for our arguments.

More precisely we say that strong local Torelli holds for the polarized manifold $(X,L)$, if there exists a holomorphic subbundle $\mathcal{H}$ of the Hodge bundle
$$\bigoplus_{k=1}^{n}\text{Hom}(\mathscr{F}^k/\mathscr{F}^{k+1}, \mathscr{F}^{k-1}/\mathscr{F}^k)$$
on the period domain ${D}$, such that the holomorphic tangent map of the period map induces an isomorphism from the tangent bundle  $\mathrm{T}^{1,0}\T$ to the Hodge subbundle $\mathcal{H}$ on $\T$

\begin{equation}\label{int local assumption}
d\Phi:\, \mathrm{T}^{1,0}\T\stackrel{\sim}{\longrightarrow} \mathcal{H}.
\end{equation}
Here we still denote by $\mathcal{H}$ the pull-back of Hodge subbundle on $\T$ for convenience.

Strong local Torelli is satisfied for various examples, including Calabi--Yau manifolds, hyperk\"ahler manifolds, some hypersurfaces in projective space or weighted projective space, and certain complete intersections in complex projective space.
We will show that they all satisfy strong local Torelli, and have explicit  Hodge subbundle $\mathcal{H}$ as required in \eqref{int local assumption} above.
Among them, we should mention the smooth hypersurfaces of degree $d$ in $\mathbb{P}^{n+1}$ satisfying $$d|(n+2)\text{ and }d\ge 3.$$ Till now, generic Torelli theorem for such hypersurfaces is still open, except the case of quintic threefold which is proved by Voisin \cite{Voisin99}.

By using eigenperiod map, we will get more non-trivial examples admitting Hodge subbundles in \eqref{int local assumption}, which contain
the arrangements of hyperplanes in $\mathbb{P}^{n}$, cubic surface and cubic threefold.

In this paper, we will prove that strong local Torelli implies global Torelli  on the Torelli space $\T'$ of the polarized manifolds in the T-class, which is the main theorem of this paper.

\begin{thm}[Main theorem]\label{intr main}
Suppose that the polarized manifold $(X,L)$ belongs to the T-class, and that strong local Torelli holds for $(X,L)$, then the global Torelli theorem holds on the Torelli space $\T'$, i.e. the period map $\Phi':\, \T'\to D$ is injective.
\end{thm}

Our proof mainly uses the construction of certain natural completion space of the Torelli space, which we will prove to be the universal cover of the Griffiths completion of the moduli space $\Z$, the affine structure on the Teichm\"uller space and
the boundedness of the period maps that we proved in \cite{LS}.

The above theorem has several interesting corollaries.

\begin{corollary}\label{intr case 1}
Let the polarized manifold $(X,L)$ be one of the following cases,
\begin{enumerate}
\item[(i)] K3 surface;
\item[(ii)] Calabi-Yau manifold;
\item[(iii)] hyperk\"ahler manifold.
\end{enumerate}
Then global Torelli theorem holds on the Torelli space $\T'$ for $(X,L)$.

\end{corollary}

Global Torelli theorem for K3 surfaces is proved in algebraic and K\"ahlerian cases in many famous papers, see \cite{BR}, \cite{Fri}, \cite{LooPet}, \cite{PSS} and so on.
In \cite{Verbitsky}, Verbitsky proved the global Torelli theorem for hyperk\"ahler manifolds on the irreducible components of the birational Teichm\"uller space and the birational moduli space as defined \cite{Verbitsky}.
Corollary \ref{intr case 1} in case (i) and (iii) can be viewed as a different version of their theorems on the Torelli space as defined in our paper.  Corollary \ref{intr case 1} in case (ii) was first proved in \cite{CLS13}.

\begin{corollary}\label{intr case 3}
Let the polarized manifold $(X,L)$ be one of the following cases,
\begin{enumerate}
\item[(i)] the smooth hypersurface of degree $d$ in $\mathbb{P}^{n+1}$ satisfying $d|(n+2)$ and $d\ge 3$;
\item[(ii)] the polarized manifold in $\mathbb{P}^{m-1}$ associated to an arrangement of $m$ hyperplanes in $\mathbb{P}^{n}$ with $m\ge n$, defined as  in \eqref{variety of hyperplane} of Section \ref{eigenperiods};
\item[(iii)] smooth cubic surface or cubic threefold.
\end{enumerate}
Then global Torelli theorem holds on the corresponding Torelli space $\T'$ for $(X,L)$.

\end{corollary}

We should mention that the Hodge subbundles satisfying \eqref{int local assumption} in case (ii) and (iii) of Corollary \ref{intr case 3} are constructed by the eigenperiod maps, see Example \ref{Hypersurfaces} and \ref{arrangements}. Also note that the moduli space in case (ii) of Corollary \ref{intr case 3}  is precisely the moduli space of hyperplane arrangements as discussed in \cite{DK}. See Section 2.2 for more details.

We remark that case (i) of Corollary \ref{intr case 3} is interesting and new.  Voisin in \cite{Voisin99} proved the generic Torelli theorem on the moduli space of quintic threefold. Case (ii) of Corollary \ref{intr case 3} with $n=1$ can be considered as a version on the Torelli space of the main result of \cite{DM}, which is the famous Deligne-Mostow theory. The general case for (ii) in Corollary \ref{intr case 3} with $n>1$ is new.

In \cite{ACT02} and \cite{ACT11}, Allcock, Carlson and Toledo proved the global Torelli theorem for cubic surfaces and cubic threefolds respectively, as one of their main results. They proved that the refined period map on the moduli space of smooth cubic surfaces or cubic threefolds is an isomorphism of complex analytic orbifolds onto its image, and the refined period map on the moduli space of framed smooth cubic surfaces or cubic threefolds is an isomorphism of complex manifolds onto its image.
Case (iii) of Corollary \ref{intr case 3} can be considered as a different version of their results in \cite{ACT02} and \cite{ACT11} for the ordinary period map on the Torelli space.


Let $\sZ_{m}^H$ denote the Hodge metric completion of the moduli space $\Z$ of polarized manifolds  with level $m$ structure with respect to the induced Hodge metric on $\Z$, and $\T^H$ be the universal cover of $\sZ_{m}^H$.

We will prove that $\T^H$ is the Hodge metric completion of the Torelli space $\T'$.
Moreover we will prove  the following theorem.

\begin{thm}\label{intr-pseudoconvex}
Suppose that the polarized manifold $(X,L)$ belongs to the T-class and the strong local Torelli holds for $(X,L)$. Then the Hodge metric completion $\T^H$ of the Torelli space $\T'$ is a bounded pseudoconvex
domain in $\C^{N}$. In particular, there exists a unique complete
K\"ahler-Einstein metric on $\T^{H}$ with Ricci curvature $-1$.
\end{thm}


Let $\Gamma=\rho(\pi_1(\ZZ))$ denote the global monodromy group,  where $$\rho:\, \pi_1(\ZZ) \to \text{Aut}(H_{\mathbb Z}, Q)$$ is the monodromy representation. As an
application, we will prove that the global Torelli theorem holds on the moduli space with certain level structure. More precisely, we have the following theorem.

\begin{theorem}\label{intr generic}
Suppose that the polarized manifold $(X,L)$ belongs to the T-class and the strong local Torelli holds for $(X,L)$.  Then the extended period map $\Phi_{\sZ_{m}^{H}}:
\,\sZ_{m}^{H}\rightarrow D/\Gamma$ is injective. As a consequence, the global Torelli theorem holds on the moduli space $\sZ_{m}$ with level $m$ structure.
\end{theorem}

More applications of affine structures on the Teichm\"uller and Torelli spaces can be found in \cite{LS1}.

The paper is divided into the following sections, which we will describe briefly. In Section \ref{Hodge-Teich}  we review the basics of Hodge theory and introduce the
definitions of various moduli spaces and period maps. In Section \ref{local Torelli}, we first define the notion of strong local Torelli. Then we discuss various examples and
study the explicit forms of the Hodge subbundle $\mathcal{H}$ as required in \eqref{int local assumption} for them.  In Section \ref{compactifications}  we discuss the three
completions of the moduli spaces and prove their equivalence. Extended period maps are also introduced in this section.

In Section \ref{bounded periods}, we review certain results in \cite{LS}, which include the boundedness of the images of  the period maps and the existence of affine structures
on the Teichm\"uller spaces for polarized manifolds satisfying strong local Torelli. In Section \ref{global Torelli}, we prove the main result of this paper, Theorem \ref{intr
main}, and apply it to some  examples  in Corollary \ref{intr case 1} and Corollary \ref{intr case 3}. We also prove Theorem \ref{intr-pseudoconvex} at the end of this section.
Section \ref{general} contains the proof of Theorem \ref{intr generic} as an application of the main result of this paper.

\vspace{0.5cm}
\textbf{Acknowledgement.}
We thank many colleagues and friends for many fruitful discussions. The research of K. Liu is supported by an NSF grant.


\section{Moduli spaces and Hodge theory}\label{Hodge-Teich}
In this section, we introduce the notions of moduli spaces, level $m$ structure,  Teichm\"uller spaces and Torelli spaces. We also review Hodge theory and define the period maps from the moduli spaces, Teichm\"uller spaces and Torelli spaces. We present the review partially as given in \cite{LS} for the reader's convenience. All the results in this section are standard and well-known. For example, one can refer to \cite{Popp} for the knowledge of moduli space, and \cite{KM}, \cite{SU} for the knowledge of deformation theory.

\subsection{Analytic families}\label{Teich}

Let $X$ be a projective manifold with an ample line bundle $L$. We call the pair $(X,L)$ a polarized manifold.

An analytic family of polarized manifolds is a proper morphism $f :\, \X\to S$ between complex analytic spaces with the following properties
\begin{itemize}
\item[(1)] the complex analytic spaces $\X$ and $S$ are smooth and connected, and the morphism $f$ is nondegenerate, i.e. the tangent map $df$ is of maximal rank at each point of $\X$;
\item[(2)] there is a line bundle $\mathcal{L}$ over $\X$;
\item[(3)] $(X_s=f^{-1}(s),L_s=\mathcal{L}|_{X_s})$ is a polarized manifold for any $s\in S$.
\end{itemize}

An analytic family $f:\, \X\to S$ is called universal at a point $s\in S$ if the following two conditions are satisfied:
\begin{itemize}
\item[(1)] Given an analytic family $f':\, \X'\to S'$ with $s'\in S'$ such that $\iota:\, f'^{-1}(s')\to f^{-1}(s)$ is biholomorphic, then there exists a neighborhood $U\subseteq S'$ around $s'$ and holomorphic maps $g :\, U\to S$ and $h :\, f'^{-1}(U)\to \X$ such that the following diagram commutes
\[ \begin{CD}
f'^{-1}(U) @>h>> \X \\
@V{f'} VV @Vf VV \\
U @>g >> S
\end{CD} \]
where $g(s')=s$, and for any $t'\in U$ with $t=g(t')$, the restricted map $$h|_{f'^{-1}(t')}:\, f'^{-1}(t')\to f^{-1}(t)$$ is biholomorphic with $h|_{f'^{-1}(s)}$ identified with the biholomorphic map $\iota$.
\item[(2)] The map $g$ is uniquely determined.
\end{itemize}

A universal family is an analytic family $f:\, \X \to S$ which is universal at every point in $S$.


\subsection{Moduli, Teichm\"uller and Torelli spaces}\label{moduli and Teichmuller}
Let $(X,L)$ be a polarized manifold.
The moduli space $\M$ of polarized manifolds is the complex analytic space parameterizing the isomorphism class of polarized manifolds with the isomorphism defined by
$$(X,L)\sim (X',L')\EQ \exists \ \text{biholomorphic map} \ f : X\to X' \ \text{s.t.} \ f^*L'=L .$$

We fix a lattice $\Lambda$ with a pairing $Q_{0}$, where $\Lambda$ is isomorphic to $H^n(X_{0},\mathbb{Z})/\text{Tor}$ for some $X_{0}$ in $\M$ and $Q_{0}$ is defined by the cup-product.
For a polarized manifold $(X,L)\in \M$, we define a marking $\gamma$ as an isometry of the lattices
\begin{equation}\label{marking}
\gamma :\, (\Lambda, Q_{0})\to (H^n(X,\mathbb{Z})/\text{Tor},Q).
\end{equation}

For any integer $m\geq 3$, we follow the definition of Szendr\"oi \cite{sz}
 to define an $m$-equivalent relation of two markings on $(X,L)$ by
$$\gamma\sim_{m} \gamma' \text{ if and only if } \gamma'\circ \gamma^{-1}-\text{Id}\in m \cdot\text{End}(H^n(X,\mathbb{Z})/\text{Tor}),$$
and denote $[\gamma]_{m}$ to be the set of all the $m$-equivalent classes of $\gamma$.
Then we call $[\gamma]_{m}$ a level $m$
structure on the polarized manifold $(X,L)$.

Two polarized manifolds with level $m$ structure $(X,L,[\gamma]_{m})$ and $(X',L',[\gamma']_{m})$ are said to be isomorphic, or equivalent, if there exists a biholomorphic map $f : X\to X'$ such that
$$f^*L'=L \text{ and } f^*\gamma'\sim_{m} \gamma,$$
where $f^*\gamma'$ is given by $$\gamma':\, (\Lambda, Q_{0})\to (H^n(X',\mathbb{Z})/\text{Tor},Q)$$ composed with the induced map
$$f^*:\, (H^n(X',\mathbb{Z})/\text{Tor},Q)\to (H^n(X,\mathbb{Z})/\text{Tor},Q).$$
We denote by $[M, L, [\gamma]_{m}]$ the isomorphism class of the polarized manifolds with level $m$ structure $(M, L, [\gamma]_{m})$.

The moduli space of polarized manifolds with level $m$ structure is the analytic space which parameterizes the isomorphism class of polarized manifolds with level $m$ structure, where $m\ge 3$.
Let $\mathscr{L}_{m}$ be the moduli space of polarized manifolds with level $m$ structure, $m\ge 3$, which contains the given polarized manifold $(X,L)$.

\begin{definition}\label{T-class} A polarized manifold  $(X,L)$ is said to belong to the T-class, if the irreducible component $\Z$ of $\mathscr{L}_{m}$ containing $(X,L)$ is a
complex manifold, on which there is an analytic family $f_{m}:\,\U_{m}\to \Z$ for all $m\ge m_{0}$, where $m_{0}\ge 3$ is some integer.
\end{definition}
In this case, we may simply take $m_{0}=3$ without loss of generality.

In this paper we only consider the polarized manifold $(X,L)$ belonging to the T-class, and the smooth moduli space $\Z$ for $m\geq 3$. Our purpose of introducing this notion is for us to  work on the smooth covers of the moduli spaces.


Let $\T^m$ be the universal covering of $\Z$ with covering map $\pi_m:\, \T^m\to \Z$. Then we have an analytic family $g_m:\, \V^m \to \T^m$ such that the following diagram is cartesian
\[ \begin{CD}
\V^m @> >> \U_m \\
@V Vg_mV  @V Vf_mV \\
\T^m @> >> \Z
\end{CD} \]
i.e. $\V^m=\U_m \times_{\Z}\T^m$. The family $g_m$ is called the pull-back family.

The proof of the following lemma is obvious.
\begin{lemma}
The space $\T^m$ is a connected complex manifold with the pull-back family $g_m:\, \V^m \to \T^m$.
\end{lemma}

The following lemma proves that the space $\T^m$ is independent of the level $m$.
\begin{lemma}\label{independent of m}
The space $\T^m$ does not depend on the choice of $m$. From now on, we simply denote $\T^m$ by $\T$, the analytic family by $g:\, \V \to \T$ and the covering map by $\pi_m:\, \T\to \Z$.
\end{lemma}
\begin{proof}  We give two proofs of this lemma.
The first proof  uses the construction of moduli space with level $m$ structure, see Lecture 10 of \cite{Popp}, or pages 692 -- 693 of \cite{sz}.

Let $m_1$ and $m_2$ be two different integers, and
$$f_{m_{1}}:\,\U_{m_1}\to \mathcal{Z}_{m_1}, \ f_{m_{2}}:\,\U_{m_2}\to \mathcal{Z}_{m_2}$$ be two analytic families with level $m_1$ structure and level $m_2$ structure respectively.
Let $\T^{m_{1}}$ and $\T^{m_{2}}$ be the universal covering space of $\mathcal{Z}_{m_1}$ and $\mathcal{Z}_{m_2}$ with the pull back family
$$g_{m_{1}}:\,\V^{m_1}\to \mathcal{\T}^{m_1}, \ g_{m_{2}}:\,\V^{m_2}\to \mathcal{\T}^{m_2}$$
of $f_{m_{1}}$ and $f_{m_{2}}$ respectively.
Let $m=m_1m_2$ and consider the analytic family $$f_{m}:\,\U_{m}\to \mathcal{Z}_{m}.$$
From the discussion in Page 130 of \cite{Popp} or 692 -- 693 of \cite{sz}, we know that $\mathcal{Z}_{m}$ is a covering space of both $\mathcal{Z}_{m_1}$ and $\mathcal{Z}_{m_2}$, and the analytic family $f_{m}$ over $\Z$ is the pull-back family of both $f_{m_{1}}$ and $f_{m_{2}}$ via the corresponding covering maps.

Let $\T$ be the universal covering space of $\mathcal{Z}_{m}$ with the pull-back family
$$g:\,\V\to \mathcal{\T}$$ of $f_{m}$. Since $\mathcal{Z}_{m}$ is a covering space of both $\mathcal{Z}_{m_1}$ and $\mathcal{Z}_{m_2}$, we conclude that $\T$ is the universal cover of both $\mathcal{Z}_{m_1}$ and $\mathcal{Z}_{m_2}$, i.e. $$\T^{m_1}\simeq \T^{m_2} \simeq \mathcal{T},$$
and that the analytic family $g$ is identified with the analytic families $g_{m_{1}}$ and $g_{m_{2}}$ via the above isomorphisms.

If the analytic family $f_m:\, \U_m \to \Z$ is universal, as defined in Section \ref{Teich}, then we have a second proof.

 Let $m_1, m_2$ be two different integers
$\ge3$, and let $\T^{m_1}$ and $\T^{m_2}$ be the corresponding two
Teichm\"uller spaces with the universal families $$g_{m_1}:\,
\V^{m_1} \to \T^{m_1} ,\  g_{m_2}:\, \V^{m_2} \to \T^{m_2}$$
respectively. Then for any point $p\in \T^{m_1}$ and the fiber
$X_p=g_{m_1}^{-1}(p)$ over $p$, there exists $q\in \T^{m_2}$ such
that $Y_q=g_{m_2}^{-1}(q)$ is biholomorphic to $X_p$. By the
definition of universal family, we can find a local neighborhood
$U_p$ of $p$ and a holomorphic map $h_p:\, U_p\to \T^{m_2}$,
$p\mapsto q$ such that the map $h_p$ is uniquely determined.

Since $\T^{m_1}$ is simply-connected, all the local holomorphic maps $$\{ h_p:\, U_p\to \T^{m_2},\, p\in \T^{m_1}\}$$ patches together to give a global holomorphic map $h:\,
\T^{m_1}\to \T^{m_2}$ which is well-defined. Moreover $h$ is unique since it is unique on each local neighborhood of $\T^{m_1}$. Similarly we have a holomorphic map $h':\,
\T^{m_2}\to \T^{m_1}$ which is also unique. Then $h$ and $h'$ are inverse to each other by the uniqueness of $h$ and $h'$. Therefore $\T^{m_1}$ and $\T^{m_2}$ are
biholomorphic.
\end{proof}

Due to the above lemma, we can denote $\T=\T^m$ for any $m\ge 3$, and call $\T$ the Teichm\"uller space of polarized manifolds.

We  define the Torelli space $\mathcal{T}'$ to be  an irreducible component of the complex analytic space consisting of biholomorphically equivalent triples of $(X, L,
\gamma)$, where $\gamma$ is a marking defined in \eqref{marking}. To be more precise, for two triples $(X, L, \gamma)$ and $(X', L', \gamma')$, if there exists a biholomorphic
map $$f:\,X\to X'$$ such that $f^*L'=L$ and
\begin{align}\label{marking identify}
f^*\gamma' =\gamma,
\end{align}
where $f^*\gamma'$ is given by $$\gamma':\, (\Lambda, Q_{0})\to (H^n(X',\mathbb{Z})/\text{Tor},Q)$$ composed with
$$f^*:\, (H^n(X',\mathbb{Z})/\text{Tor},Q)\to (H^n(X,\mathbb{Z})/\text{Tor},Q),$$
then $[X, L, \gamma]=[X', L', \gamma']\in \mathcal{T}'$, where $[X, L, \gamma]$ is the equivalent class of the triple $(X, L, \gamma)$.

To summarize, we have the definition of the Torelli space $\T'$ as follows.
\begin{definition}
The Torelli space $\T'$  is an irreducible component of the moduli space of marked and polarized manifolds containing $(X, L)$ in the T-class.
\end{definition}

By construction, the Torelli space $\T'$ is a covering space of $\mathcal{Z}_m$, with the natural covering map $\pi'_m:\, \mathcal{T}'\rightarrow\mathcal{Z}_m$ given by
$$[X, L, \gamma]\mapsto [X, L, [\gamma]_{m}].$$
Then we have the pull-back family $g':\, \mathcal{U}'\rightarrow \mathcal{T}'$ of $f_{m}:\, \mathcal{U}_{m}\rightarrow \mathcal{Z}_m$.

\begin{proposition}\label{imp}
For the polarized manifold $(X, L)$ in the T-class,  the Torelli space $\mathcal{T}'$ containing $(X, L)$ is a connected complex
manifold.
\end{proposition}
\begin{proof}One notices that there is a natural covering map $\pi'_m: \,\mathcal{T}'\rightarrow \mathcal{Z}_m$ for any $m\geq 3$
according to the definition of $\mathcal{T}'$. Therefore $\mathcal{T}'$ is a connected complex manifold, as $\mathcal{Z}_m$ is a connected complex manifold by the definition of
T-class.
\end{proof}

By the universal property of the universal covering space, we know that the Teichm\"uller space $\T$ is also the universal cover of the Torelli space $\T'$.


\subsection{Variation of Hodge structure}\label{Hodge}
Let $(X,L)$ be polarized manifold defined in Section \ref{Teich} with $\dim_\C X=n$.
The $n$-th primitive cohomology groups $H_{pr}^n(X,\C)$ of $X$ is defined as
$$H_{pr}^n(X,\C)=\ker\{L:\, H^n(X,\C)\to H^{n+2}(X,\C)\},$$
where the action of $L$ on the cohomology group is defined by the wedge product with the first Chern class of $L$.

Let us denote
$$H_\mathbb{Z}=H_{pr}^n(X)\cap H^n(X,\mathbb{Z})$$
and denote
$$H_{pr}^{k,n-k}(X)=H^{k,n-k}(X)\cap H_{pr}^n(X,{\mathbb{C}})$$
with its complex dimension denoted by $h^{k,n-k}$.
Since $L$ is defined over $\mathbb{Z}$, we have that $$H_{pr}^n(X,\C)=H_\mathbb{Z}\otimes_{\mathbb{Z}} \C.$$
We then have the Hodge decomposition
\begin{align}  \label{hod}
H_{pr}^n(X,{\mathbb{C}})=H_{pr}^{n,0}(X)\oplus\cdots\oplus
H_{pr}^{0,n}(X),
\end{align}
such that $H_{pr}^{n-k,k}(X)=\bar{H_{pr}^{k,n-k}(X)}$.
Therefore the data $$\{H_\mathbb{Z},H_{pr}^{k,n-k}(X)\}$$ is a Hodge structure of weight $n$.

On this Hodge structure there is a bilinear form
$$Q:\, H_\mathbb{Z}\times H_\mathbb{Z} \to \mathbb{Z}$$
defined by
\begin{equation*}
Q(u,v)=(-1)^{\frac{n(n-1)}{2}}\int_X u\wedge v
\end{equation*}
for any $d$-closed $n$-forms $u,v$ on $X$.
It is well-known that $Q $ is nondegenerate and can be extended to
$H_{pr}^n(X,{\mathbb{C}})$ bilinearly. Moreover, it also satisfies the
Hodge-Riemann relations
\begin{eqnarray}
Q\left ( H_{pr}^{n-k,k}(X), H_{pr}^{n-l,l}(X)\right )=0\text{ unless }k+l=n;\label{HR1}\\
\left (\sqrt{-1}\right )^{2k-n}Q\left ( v,\bar v\right )>0\text{ for }v\in H_{pr}^{k,n-k}(X)\setminus\{0\}. \label{HR2}
\end{eqnarray}

Let $f^k=\sum_{i=k}^nh^{i,n-i}$, $m=f^0$, and
$$F^k=F^k(X)=H_{pr}^{n,0}(X)\oplus\cdots\oplus H_{pr}^{k,n-k}(X),$$
from which we have the decreasing filtration
$$H_{pr}^n(X,\C)=F^0\supset\cdots\supset F^n.$$ We know that
\begin{align}
& \dim_{\mathbb{C}} F^k=f^k, \label{periodcondition} \\
& H^n_{pr}(X,{\mathbb{C}})=F^{k}\oplus \bar{F^{n-k+1}},\text{ and }H_{pr}^{k,n-k}(X)=F^k\cap\bar{F^{n-k}}.\nonumber
\end{align}
In terms of Hodge filtrations, the Hodge-Riemann relations \eqref{HR1} and \eqref{HR2} are
\begin{align}
& Q\left ( F^k,F^{n-k+1}\right )=0;\label{HR1'}\\
& Q\left ( Cv,\bar v\right )>0\text{ if }v\ne 0,\label{HR2'}
\end{align}
where $C$ is the Weil operator given by $$Cv=\left (\sqrt{-1}\right
)^{2k-n}v\text{ for }v\in H_{pr}^{k,n-k}(X).$$ The period domain $D$
for polarized Hodge structures is defined by the space of all such Hodge filtrations
\begin{equation*}
D=\left \{ F^n\subset\cdots\subset F^0=H_{pr}^n(X,{\mathbb{C}})\mid
\eqref{periodcondition}, \eqref{HR1'} \text{ and } \eqref{HR2'} \text{ hold}
\right \}.
\end{equation*}
The compact dual $\check D$ of $D$ is defined by
\begin{equation*}
\check D=\left \{ F^n\subset\cdots\subset F^0=H_{pr}^n(X,{\mathbb{C}})\mid
\eqref{periodcondition} \text{ and } \eqref{HR1'} \text{ hold} \right \}.
\end{equation*}
The compact dual $\check D$ is a projective manifold and the period domain $D\subseteq \check D$ is an open submanifold.

From the definition of period domain we naturally get the Hodge bundles on $\check{D}$, by associating to each point in $\check{D}$ the vector spaces $\{F^k\}_{k=0}^n$ in the Hodge filtration of that point. We will denote
the Hodge bundles by $\mathscr{F}^k \to \check{D}$ with $\mathscr{F}^k|_{p}=F^k_{p}$ as the fiber for any $p\in \check{D}$ and each $0\leq k\leq n$.

For the family $f_m : \U_m \to \Z$, we denote each fiber by $$[X_s,L_{s},[\gamma_{s}]_{m}]=f_m^{-1}(s)$$ and $F^k_s=F^k(X_{s})$ for $0\le k\le n$.  With some fixed point
$s_0\in \Z$, the period map is defined as a morphism $$\Phi_{\Z} :\Z \to D/\Gamma$$ by
\begin{equation}\label{perioddefinition}
s \mapsto \tau^{[\gamma_s]}(F^n_s\subseteq\cdots\subseteq F^0_s)\in D,
\end{equation}
where $F_s^k=F^k(X_s)$ and $\tau^{[\gamma_s]}$ is an isomorphism between $\C-$vector spaces
$$\tau^{[\gamma_s]}:\, H^n(X_s,\mathbb{C})\to H^n(X_{s_0},\mathbb{C}),$$
which depends only on the homotopy class $[\gamma_s]$ of the curve $\gamma_s$ between $s$ and $s_0$.

Recall that the monodromy group $\Gamma$ is the image of representation of the fundamental group of $\Z$ in $\text{Aut}(H_{\mathbb{Z}},Q)$, the automorphism group of $H_{\mathbb Z}$ preserving $Q$. Then the period map from $\Z$ is well-defined with respect to the monodromy representation
$$\rho : \pi_1(\Z)\to \Gamma \subseteq \text{Aut}(H_{\mathbb{Z}},Q).$$It is well-known that the period map has the following properties:
\begin{enumerate}
\item locally liftable;
\item holomorphic, i.e. $\partial F^i_z/\partial \bar{z}\subset F^i_z$, $0\le i\le n$;
\item Griffiths transversality: $\partial F^i_z/\partial z\subset F^{i-1}_z$, $1\le i\le n$.
\end{enumerate}

Let $X$ be any polarized manifold with local coordinate chart $(U; w)$, and $\theta\in H^{1}(X,\Theta_{X})$ with the local representation $$\theta=\sum_{i\bar{j}}\theta_{i\bar{j}}  \frac{\d}{\d w^{i}}\otimes d\bar{w}^{j}$$ under the identification
$$H^{1}(X,\Theta_{X})\simeq \frac{Z^{0,1}(X, \mathrm{T}^{1,0}X)}{\bar{\d}(A^{0}(X, \mathrm{T}^{1,0}X))},$$
where $Z^{0,1}(X, \mathrm{T}^{1,0}X)$ denotes the space of all the closed $(0,1)$ forms on $X$ with values in $\mathrm{T}^{1,0}X$, and $A^{0}(X, \mathrm{T}^{1,0}X)$ denotes the space of the smooth sections of $\mathrm{T}^{1,0}X$ on $X$.

For any $(p, q)$ form $\alpha=\sum_{|I|=p,|J|=q}a_{I,\bar{J}}dw^{I}\wedge d\bar{w}^{J}$, the contraction map $\theta\lrcorner$ is given by
\begin{align}\label{contraction}
\theta\lrcorner\alpha=\sum_{k=1}^{p}\sum_{i_{1},\cdots,i_{p} \atop |J|=q,\bar{j}}(-1)^{k-1}\theta_{i_{k}\bar{j}}a_{i_{1},\cdots,i_{p},\bar{J}}d\bar{w}^{j}\wedge dw^{i_{1}}\cdots \widehat{dw^{i_{k}}} \cdots  dw^{i_{p}} \wedge d\bar{w}^{J}.
\end{align}
Hence the contraction map $\theta\lrcorner$ maps $H^{p,q}(X)$ to $H^{p-1,q+1}(X)$.

From \cite{Griffiths2}, we know that the derivative of the period map is precisely given by the contraction $KS(v)\lrcorner$, which is a linear  map. Here
$$KS: \, \mathrm{T}^{1,0}_{q}\Z\to H^{1}(X_{q},\Theta_{X_{q}})$$
is the Kodaira-Spencer map and $v\in \mathrm{T}_{q}\Z$ is a tangent vector at any point $q\in \Z$.

By (1) of the properties of the period map as above, we can lift the period map onto the universal cover $\T$ of $\Z$, to get the lifted period map $\Phi :\, \T \to D$ such that the diagram
$$\xymatrix{
\T \ar[r]^-{\Phi} \ar[d]^-{\pi_m} & D\ar[d]^-{\pi_{D}}\\
\Z \ar[r]^-{\Phi_{\Z}} & D/\Gamma
}$$
is commutative.

From the definition of marking in \eqref{marking}, we have a well-defined period map $\Phi':\, \T'\to D$ from the Torelli space $\T'$ by
\begin{equation}\label{defn of P'}
p \mapsto \gamma_{p}^{-1}(F^n_p\subseteq\cdots\subseteq F^0_p)\in D,
\end{equation}
where the triple $[X_{p}, L_{p}, \gamma_{p}]$ is the fiber over $p\in \T'$ of the analytic family $\mathcal{U}'\to \T'$, and the marking $\gamma_{p}$ is an isometry from a fixed lattice $\Lambda$ to $H^{n}(X_{p},\mathbb{Z})/\text{Tor}$, which extends $\C$-linearly to an isometry from $H=\Lambda\otimes_{\mathbb{Z}}\C$ to $H^{n}(X_{p},\mathbb{C})$.
Here
$$\gamma_{p}^{-1}(F^n_p\subseteq\cdots\subseteq F^0_p)=\gamma_{p}^{-1}(F^n_p)\subseteq\cdots\subseteq\gamma_{p}^{-1}(F^0_p)=H$$
denotes a Hodge filtration of $H$.

Now we summarize the period maps defined as above in the following commutative diagram
$$\xymatrix{
\T \ar[dr]^-{\pi}\ar[rr]^-{\Phi} \ar[dd]^-{\pi_m} && D\ar[dd]^-{\pi_{D}}\\
&\T'\ar[ur]^-{\Phi'}\ar[dl]_-{\pi'_{m}}&\\
\Z \ar[rr]^-{\Phi_{\Z}} && D/\Gamma,
}$$
where the  maps $\pi_{m}$, $\pi'_{m}$ and $\pi$ are all natural covering maps between the corresponding spaces as discussed in Section \ref{Hodge-Teich}.

 Before closing this section, we prove a lemma concerning the monodromy group $\Gamma$ on $\Z$ for $m\geq 3$.
\begin{lemma}\label{trivial monodromy}
Let $\gamma$ be the image of some element of $\pi_1(\mathcal{Z}_m)$ in $\Gamma$ under the monodromy representation. Suppose that $\gamma$ is finite, then $\gamma$ is trivial. Therefore for $m\geq 3$, we can assume that $\Gamma$ is torsion-free and $D/\Gamma$ is smooth.
\end{lemma}
\begin{proof}
Let us look at the period map locally as $\Phi_{\mathcal{Z}_m}:\,\Delta^* \to D/\Gamma$. Assume that $\gamma$ is the monodromy  action corresponding to the generator of the
fundamental group of $\Delta^*$.

We lift the period map to $\Phi :\, {\mathbb H}\to D$, where ${\mathbb H}$ is the upper half plane and the covering map from ${\mathbb H}$ to $\Delta^*$ is
$$z\mapsto \exp(2\pi \sqrt{-1} z).$$
Then $\Phi(z+1)=\gamma\Phi(z)$ for any $z\in {\mathbb H}$. Since $\Phi(z+1)$ and $\Phi(z)$ correspond to the same point in $\mathcal{Z}_m$, by the definition of $\mathcal{Z}_m$ we have
$$\gamma\equiv \text{I} \text{ mod }(m).$$
But $\gamma$ is also in $\text{Aut}(H_\mathbb{Z})$, applying Serre's lemma \cite{Serre60} or Lemma 2.4 in \cite{sz}, we have $\gamma=\text{I}$.
\end{proof}


\section{Strong local Torelli and Examples}\label{local Torelli}

In this section, we first introduce the notion of strong local Torelli, which is given in Definition \ref{local-Torelli} below. In particular we remark that the identification \eqref{local assumption} below of the tangent bundle of the Teichm\"uller space to certain Hodge subbundle is crucial for our proof of global Torelli theorem from strong local Torelli.

Then we give some examples including hyperk\"ahler manifolds, Calabi--Yau manifolds, some hypersurfaces in projective spaces or weighted projective spaces, and some complete intersections in complex projective space.
They all satisfy strong local Torelli, and have explicit forms of the Hodge subbundle $\mathcal{H}$ in \eqref{local assumption}.
We also introduce the notion of eigenperiod map to get more nontrivial examples admitting Hodge subbundles in \eqref{local assumption}, which contain
the arrangements of hyperplanes in $\mathbb{P}^{n}$, cubic surface and cubic threefold.

\subsection{Strong local Torelli and examples}

We first introduce a stronger notion of local Torelli, which we will prove to be a sufficient condition for global Torell on the Torelli spaces for projective manifolds in the T-class. Let $\Z$ be the smooth component of moduli space of polarized manifolds with  level $m$ structure containing $(X, L)$,  and let $\T$ be the corresponding Teichm\"uller space.

\begin{definition}\label{local-Torelli}
The period map $\Phi:\, \T \to D$ is said to satisfy the strong local Torelli, provided that the local Torelli  holds,  i.e. the tangent map of the period map $\Phi$ is
injective everywhere on $\T$, and furthermore, there exists a holomorphic subbundle $\mathcal{H}$ of the Hodge bundle
$$\bigoplus_{k=1}^{n}\text{Hom}(\mathscr{F}^k/\mathscr{F}^{k+1}, \mathscr{F}^{k-1}/\mathscr{F}^k)$$
on the period domain $D$, such that the tangent map of the period map induces an isomorphism on $\T$ of the holomorphic  tangent bundle $\mathrm{T}^{1,0}\T$ to the Hodge
subbundle $\mathcal{H}$,

\begin{equation}\label{local assumption}
d\Phi:\, \mathrm{T}^{1,0}\T\stackrel{\sim}{\longrightarrow} \mathcal{H}.
\end{equation}
Here we still denote by $\mathcal{H}$ the pull-back of Hodge subbundle on $\T$ for convenience.
\end{definition}

For simplicity we will also say that strong local Torelli holds for the polarized manifold $(X,L)$, if the corresponding period map $\Phi$ satisfies the above definition.

The purpose of introducing the notion of strong local Torelli is for us to extend affine structure to the Hodge metric completion space $\T^H$  to be studied in detail later.

 The following special case of condition  \eqref{local assumption} is interesting, which includes many well-known examples.
\begin{quotation}
There exist holomorphic subbundles $\mathscr{F}^l_{0}$ of $\mathscr{F}^l$, $0\le l\le n$, and there exists some integer $k$ from $0$ to $n$, such that $\mathscr{F}_{0}^l=0$ for any $l>k$ and the holomorphic subbundle $\mathcal{H}$ of Definition \ref{local-Torelli} is of the form
\begin{equation}\label{remark local assumption}
\mathcal{H}=\text{Hom}(\mathscr{F}^k_{0}, \mathscr{F}^{k-1}_{0}/\mathscr{F}_{0}^k).
\end{equation}

\end{quotation}

Now we discuss some examples satisfying the strong local Torelli to illustrate the Hodge subbundle $\mathcal{H}$ of Definition \ref{local-Torelli}.

First note that, from the arguments before Corollary \ref{case 1} and Corollary \ref{case 3} in Section \ref{global Torelli}, one sees that the moduli spaces $\Z$ for $m$ large enough, as well as the Teichm\"uller space $\T$,  are smooth in the following examples. Hence we only need to describe the Hodge subbundle $\mathcal{H}$ locally, or equivalently, from the point of view of deformation theory.

\begin{example}[K3 surfaces and Calabi--Yau manifolds]
Let $M$ be a Calabi--Yau manifold of complex dimension $n\ge 3$, which means that $M$ is a compact projective manifold with a trivial canonical bundle
$$K_{M}=\Omega_{M}^{n}\simeq \mathcal{O}_{M}$$ and satisfies $H^i(M,\mathcal{O}_M)=0$ for $0<i<n$. Then the contraction map induces an isomorphism of the sheaves
$$\Theta_{M}=\mathcal{O}_{M}(\mathrm{T}_{M})\simeq\Omega_{M}^{n}(\mathrm{T}_{M})\simeq \Omega_{M}^{n-1}.$$
Hence $$H^{1}(M,\Theta_{M})\simeq H^{1}(M,\Omega_{M}^{n-1})\simeq H^{n-1,1}(M),$$which together with the unobstructedness of the deformations of compact K\"ahler manifolds with trivial canonical bundle implies that local Torelli theorem holds for Calabi--Yau manifolds such that
$$H^{1}(M,\Theta_{M})\stackrel{\sim}{\longrightarrow} \text{Hom}(F^{n}, F^{n-1}/F^{n}),$$
is an isomorphism.
Hence the subbundle $\mathcal{H}$ of Definition \ref{local-Torelli} exists and is
$$\mathcal{H}=\text{Hom}(\mathscr{F}^n, \mathscr{F}^{n-1}/\mathscr{F}^n).$$
Therefore Calabi--Yau manifolds satisfy Definition \ref{local-Torelli}. In fact, global Torelli theorem on the Torelli space for Calabi--Yau manifolds was already proved in \cite{CLS13}.

Since K3 surfaces are $2$-dimensional analogues of Calabi--Yau manifolds, we can see that K3 surfaces
also satisfy Definition \ref{local-Torelli} by a similar argument.
Global Torelli theorem for K3 surfaces are proved in algebraic and K\"ahlerian cases in many famous papers, see \cite{BR}, \cite{Fri}, \cite{LooPet}, \cite{PSS} and so on.

\end{example}

\begin{example}[Hyperk\"ahler manifolds]
Let $X$ be a hyperk\"ahler manifold, which means that $X$ is an irreducible simply-connected compact K\"ahler manifold such that $H^{0}(X,\Omega_{X}^{2})$ is generated by an
everywhere nondegenerate holomorphic two form $\sigma$.

A direct consequence of the definition of $X$ is that the cohomology $H^{2}(X,\C)$ has the following Hodge decomposition
$$H^{2}(X,\C)=H^{2,0}(X)\oplus H^{1,1}(X)\oplus H^{0,2}(X),$$
such that $H^{2,0}(X)=\C\sigma$ and $H^{0,2}(X)=\C\bar{\sigma}$. Moreover, $X$ has complex dimension $2n$ and the canonical bundle $K_{X}=\Omega_{X}^{2n}$ is trivialized by $\sigma^{n}$. The everywhere nondegenerate holomorphic two form $\sigma$ also induces an isomorphism
$$\sigma :\, \Theta_{X}=\mathcal{O}(\mathrm{T}_{X})\to \Omega_{X}.$$

The unobstructedness of the deformations of compact K\"ahler manifolds with trivial canonical bundles and the isomorphism above imply that the infinitesimal Torelli theorem holds for hyperk\"ahler manifolds such that
$$H^{1}(X,\Theta_{X})\stackrel{\sim}{\longrightarrow} \text{Hom}(F^{2}, F^{1}/F^{2}),$$
is an isomorphism.
Therefore the subbundle $\mathcal{H}$ of Definition \ref{local-Torelli} exists and is $$\mathcal{H}=\text{Hom}(\mathscr{F}^2, \mathscr{F}^{1}/\mathscr{F}^2).$$

\end{example}

\begin{example}[Hypersurfaces in projective spaces]\label{Hypersurfaces}
Let $X$ be a smooth hypersurface of degree $d$ in $\mathbb{P}^{n+1}$, which means that $X$ is an algebraic subvariety of $\mathbb{P}^{n+1}$ determined by a polynomial $F(x_{0},\cdots,x_{n+1})$ with nondegenerate Jacobian.
Define the graded Jacobian quotient ring $R(F)$ by
$$R(F)=\C[x_{0},\cdots,x_{n+1}]\Big/\Big<\frac{\partial F}{\partial x_{0}},\cdots , \frac{\partial F}{\partial x_{n+1}}\Big>.$$
As proved in \cite{Griffiths69}, the primitive cohomology of $X$ is
$$H^{k,n-k}_{pr}(X)\simeq R(F)^{d(n+1-k)-n-2},$$
where $R(F)^{l}$ is the graded piece of degree $l$ of $R(F)$, $l\ge 0$.

From Lemma 5.4.4 and Example 8.1.1 of \cite{CMP}, we have that the infinitesimal Torelli theorem holds for the smooth hypersurface $$X=(F=0)\subset \mathbb{P}^{n+1}$$ of degree
$d$ if the product
$$R(F)^{d}\times R(F)^{d(n+1-k)-n-2}\to R(F)^{d(n+2-k)-n-2}$$
is nondegenerate in the first factor for some $k$ between $1$ and $n$. Macaulay's theorem tells that this product map is nondegenerate in each factor as long as $$0 \le
d(n+2-k)-n-2 \le (n+2)(d-2).$$ Therefore the only exceptions for the infinitesimal Torelli are quadric and cubic curves, and cubic surfaces.

Let us consider the following case. Suppose there exists some integer $1\le k\le n$ such that
\begin{equation}\label{cy hypersurface}
R(F)^{d}\simeq R(F)^{d(n+2-k)-n-2}.
\end{equation}
One sufficient condition of \eqref{cy hypersurface} is that $$d=d(n+2-k)-n-2, \text{ i.e. } d(n+1-k)=n+2,$$ or equivalently $$d|(n+2).$$
Taking infinitesimal Torelli into consideration, we require that $d\le (n+2)(d-2)$, and hence $$d\ge 3.$$
In this case we have $$H^{k-1,n-k+1}_{pr}(X)\simeq R(F)^{d},\
H^{k,n-k}_{pr}(X)\simeq R(F)^{0}\simeq \C,$$
and  $H^{l,n-l}_{pr}(X)=0$ for $l>k$.
Then the Hodge subbundle $\mathcal{H}$ of Definition \ref{local-Torelli} exists and is
$$\mathcal{H}=\text{Hom}(\mathscr{F}^{k}, \mathscr{F}^{k-1}/\mathscr{F}^{k}).$$

We remark that  generic Torelli theorem for smooth hypersurface of degree $d$ in $\mathbb{P}^{n+1}$ satisfying $d|(n+2)$ is still open, except for quintic threefold, i.e. $d=5$ and $n=3$, see \cite{Voisin99}. Still our method gives the  global Torelli theorem on the Torelli space for such class of projective  manifolds.

We now look at some well-known examples satisfying $d|(n+2)$ and $d\ge 3$.

If $n=2$, then the nontrivial case is $d=4$, i.e. the quartic surface, which is a special case of K3 surface.

If $n=3$, then $d=5$, i.e. the quintic threefold, which is a special case of Calabi--Yau threefold.

If $n=4$, then the nontrivial cases are $d=3$ and $d=6$. The former is cubic fourfold and the later is a Calabi--Yau manifold.

\end{example}

Note that the condition $d|(n+2)$, $d\ge 3$ is only one of the sufficient conditions from Definition \ref{local-Torelli} in the case of hypersurfaces. We believe that many more examples of hypersurfaces satisfying Definition \ref{local-Torelli} can be found, just like cubic surfaces and cubic threefolds in Example \ref{cubic} below.

We also remark that we can generalize the consideration in Example \ref{Hypersurfaces} to other cases, including $k$-sheeted branched covering of $\mathbb{P}^{n}$, Veronese double cone and some complete intersections in complex projective space, since there are similar theories studying infinitesimal Torelli theorems in terms of the Jacobian quotient rings in these cases, see, for examples, \cite{Konno91} and \cite{Saito}.

\subsection{Eigenperiod map and examples}\label{eigenperiods}

From the examples above, we can see that in general the canonically polarized hypersurfaces in $\mathbb{P}^{n}$ seem not to satisfy strong local Torelli, since the rank of the Hodge bundles $\text{Hom}(\mathscr{F}^{k}, \mathscr{F}^{k-1}/\mathscr{F}^{k})$, $1\le k\le n$ are too big when compared with the dimension of the Teichm\"uller space $\T$. But the notion of eigenperiod map gives us many new examples in such cases.
Below we will give a brief review of the eigenperiod map. One can refer to \cite{DK} for more details.

Let $\T$ be the Teichm\"uller space of $(X, L)$, and $g:\, \U \to \T$ be the analytic family. Let $G$ be a finite abelian group acting holomorphically on $\U$, preserving the line bundle $\mathcal{L}$ on $\U$. Recall that the line bundle $\mathcal{L}$ defines a polarization $L_{q}=\mathcal{L}|_{X_{q}}$ on $X_{q}=g^{-1}(q)$ for any $q\in \T$.
We assume that $g(X_{q})=X_{q}$ for any $q\in \T$ and $g\in G$.
Fix $p\in \T$ and $o=\Phi(p)\in D$ as the base points. Let $H_{\mathbb{Z}}=H^{n}(X_{p},\mathbb{Z})/\text{Tor}$ with the Poincar\'e paring $Q$. Then the simply-connectedness of $\T$ implies that we can identify $$(H^{n}(X_{q},\mathbb{Z})/\text{Tor},Q)\stackrel{\sim}{\longrightarrow}(H_{\mathbb{Z}},Q)$$ for any $q\in \T$.
Since $G$ preserves the line bundle $\mathcal{L}$ on $\U$, we have a induced action of $G$ on $H_{\mathbb{Z}}$ preserving $Q$, i.e. we have a representation $\rho:\, G\to \text{Aut}(H_{\mathbb{Z}},Q)$.

Let $H=H_{\mathbb{Z}}\otimes \C$.
Define $D^{\rho}$ by
$$D^{\rho}=\{ (F^n\subset\cdots\subset F^0)\in D:\, \rho(g)(F^k)=F^k, 0\le k\le n, \text{ for any }g\in G \}.$$
Let $\chi\in \text{Hom}(G,\C^{*})$ be a character of $G$.
Let
$$H_{\chi}=\{v\in H:\, \rho(g)(v)=\chi(g)v, \forall g\in G\}.$$
For any $(F^n\subset\cdots\subset F^0)\in D^{\rho}$, we define $$F^{k}_{\chi}=F^{k}\cap H_{\chi}\text{ and }H^{k,n-k}_{\chi}=F^{k}\cap \bar{F}^{n-k}\cap H_{\chi},\ 0\le k\le n.$$
Then we have the decomposition
\begin{align}\label{chi decomp}
H_{\chi}=H^{n,0}_{\chi}\oplus \cdots \oplus H^{0,n}_{\chi}.
\end{align}
Note that decomposition \eqref{chi decomp} is a Hodge decomposition if and only if $\chi$ is real.

Although decomposition \eqref{chi decomp} is not Hodge decomposition for general $\chi$, it still has the restricted polarization $Q$ on $H_{\chi}$ such that
\begin{align*}
Q( Cv,\bar v)>0\text{ for any }v \in H_{\chi}\setminus \{0\}.
\end{align*}
Hence the sub-domain $D^{\rho}_{\chi}$ defined by
$$D^{\rho}_{\chi}=\{(F_{\chi}^n\subset\cdots\subset F_{\chi}^0):\,(F^n\subset\cdots\subset F^0)\in D^{\rho}\},$$
has a well-defined metric, which is the restriction of the Hodge metric on $D$.

Since the action of $G$ on any fiber $X_{q}$ is holomorphic and preserves the polarization $L_{q}$ on $X_{q}$, the period map $\Phi:\, \T \to D$ takes values in $D^{\rho}$. Then we define the eigenperiod map by
$$\Phi_{\chi}:\, \T \to D^{\rho}_{\chi},$$
which is the composition of $\Phi$ with the projection map $D^{\rho} \to D^{\rho}_{\chi}$.
We can also define the Hodge subbundles $\mathscr{F}^{k}_{\chi}$, $0\le k\le n$, on $D^{\rho}_{\chi}$, and pull them back via the eigenperiod map to get the Hodge subbundles on $\T$, which are still denoted by $\mathscr{F}^{k}_{\chi}$, $0\le k\le n$.

Historically, the eigenperiod map is an effective way to reduce the dimension of the period domain, so that the eigenperiod map could still be locally isomorphic.
In this paper we will study the the ordinary period maps $\Phi$ and $\Phi'$ on $\T$ and $\T'$ respectively, while the eigenperiod map only serves as a tool to find the Hodge subbundle in \eqref{remark local assumption}.
More precisely, we are interested in the following case, which includes many known examples.
\begin{quotation}
There exist some character $\chi\in \text{Hom}(G,\C^{*})$ and some integer $k$ from $0$ to $n$, such that $\mathscr{F}^l_{\chi}=0$ for any $l>k$, and that the Hodge subbundles
$\mathscr{F}^k_{\chi}, \mathscr{F}^{k-1}_{\chi}$ satisfy \eqref{remark local assumption}, which is equivalent to that the Hodge subbundle $\mathcal{H}$ of Definition
\ref{local-Torelli} is of the form
\begin{equation}\label{eigen local assumption}
\mathcal{H}=\text{Hom}(\mathscr{F}^k_{\chi}, \mathscr{F}^{k-1}_{\chi}/\mathscr{F}_{\chi}^k).
\end{equation}

\end{quotation}

We remark that the global Torelli theorem will be proved in the above case for the original period map $\Phi':\, \T' \to D$, and the global Torelli theorem for the eigenperiod map follows in a similar way.

\begin{example}[Cubic surfaces and cubic threefolds]\label{cubic}
We follow the notations of \cite{Debarre}.
Let $X\subset \mathbb{P}^{3}$ be a smooth cubic surface with defining equation $F(x_{0},x_{1},x_{2}, x_{3})$. Since $H^{2}(X,\C)=H^{1,1}(X)$ and the period map is trivial, we have to consider other ways to define the period map.

In their papers \cite{ACT02}, Allcock, Carlson, and Toledo consider the cyclic triple covering $\tilde{X}$ of $\mathbb{P}^3$ branched along $X$, with $\tilde{X}$ defined by $$\tilde{X}=(F(x_{0},x_{1},x_{2}, x_{3})+x_4^3=0)\subset \mathbb{P}^{4}.$$ The Hodge structure of $\tilde{X}$ is
$$H^{3}(\tilde{X},\C)=H^{2,1}(\tilde{X})\oplus H^{1,2}(\tilde{X}),$$
where $h^{2,1}(\tilde{X})=h^{1,2}(\tilde{X})=5$.
Moreover in \cite{ACT02}, they show that the Hodge structure carries an action of the group $\mu_{3}$ of cubic roots of unity, which is induced by the action of the group $\mu_{3}$ on $\tilde{X}$ such that $$[x_{0},x_{1},x_{2}, x_{3}, x_4]\mapsto [x_{0},x_{1},x_{2}, x_{3}, \omega x_4],\text{ where }\omega=\text{e}^{2\pi i/3}.$$
The eigenspace $H^{3}_{\bar{\omega}}(\tilde{X},\C)$ for the eigenvalue $\bar{\omega}=\text{e}^{\pi i/3}$ has the Hodge structure
$$H^{3}_{\bar{\omega}}(\tilde{X},\C)=H^{2,1}_{\bar{\omega}}(\tilde{X})\oplus H^{1,2}_{\bar{\omega}}(\tilde{X}),$$
where $\dim_{\C}(H^{2,1}_{\bar{\omega}}(\tilde{X}))=1$ and $\dim_{\C}(H^{1,2}_{\bar{\omega}}(\tilde{X}))=4$. In this case, the infinitesimal Torelli theorem holds in the form
$$H^{1}({X},\Theta_{{X}})\stackrel{\sim}{\longrightarrow} \text{Hom}(F_{\bar{\omega}}^{2}, F_{\bar{\omega}}^{1}/F_{\bar{\omega}}^{2}),$$
where $$F_{\bar{\omega}}^{2}=H^{2,1}_{\bar{\omega}}(\tilde{X})\text{ and }F_{\bar{\omega}}^{1}=H^{2,1}_{\bar{\omega}}(\tilde{X})\oplus H^{1,2}_{\bar{\omega}}(\tilde{X}).$$
Hence the subbundle $\mathcal{H}$ of Definition \ref{local-Torelli} exists and is
$$\mathcal{H}=\text{Hom}(\mathscr{F}_{\bar{\omega}}^{2}, \mathscr{F}_{\bar{\omega}}^{1}/\mathscr{F}_{\bar{\omega}}^2).$$

Similarly if $Y$ is a smooth cubic threefold, in \cite{ACT11}, they consider the cyclic triple covering $\tilde{Y}$ of $\mathbb{P}^4$ branched along $Y$, such that $H^{4}(\tilde{Y},\C)$ has an eigenspace $H^{4}_{\omega}(\tilde{Y},\C)$ for the eigenvalue $\omega=\text{e}^{2\pi i/3}$ with the Hodge structure
$$H^{4}_{\omega}(\tilde{Y},\C)=H^{3,1}_{\omega}(\tilde{Y})\oplus H^{2,2}_{\omega}(\tilde{Y}),$$
where $H^{3,1}_{\omega}(\tilde{Y})=H^{3,1}(\tilde{Y})$, and $\dim_{\C}(H^{3,1}_{\omega}(\tilde{Y}))=1$ and $\dim_{\C}(H^{2,2}_{\omega}(\tilde{Y}))=10$.
In this case, the infinitesimal Torelli theorem holds in the form
$$H^{1}({Y},\Theta_{{Y}})\stackrel{\sim}{\longrightarrow} \text{Hom}(F_{\omega}^{3}, F_{\omega}^{2}/F_{\omega}^{3}),$$
where $F_{\omega}^{3}=H^{3,1}_{\omega}(\tilde{Y})$ and $F_{\omega}^{2}=H^{4}_{\omega}(\tilde{Y},\C)$, and hence the subbundle $\mathcal{H}$ of Definition \ref{local-Torelli} exists and is
$$\mathcal{H}=\text{Hom}( \mathscr{F}_{\omega}^{3},  \mathscr{F}_{\omega}^{2}/ \mathscr{F}_{\omega}^{3}).$$

\end{example}

\begin{example}[Arrangements of hyperplanes]\label{arrangements}
Let $m\ge n$ be positive integers. Consider the complementary set
$$U=\mathbb{P}^n\setminus \bigcup_{i=1}^m H_i,$$
where $H_i$ is a hyperplane of $\mathbb{P}^n$ for each $1\le i\le m$ such that $H_1,\cdots ,H_m$ are in general positions. That is to say that, if $H_i$ is defined by the
linear forms
\begin{align}\label{linear form}
f_i(z_0,\cdots, z_n)=\sum_{j=0}^n a_{ij}z_j, 1\le i\le m,
\end{align}
then the matrix $(a_{ij})_{1\le i\le m,0\le j\le n}$ is of full rank.

The fundamental group $\pi_{1}(U)$ of $U$ is generated by the basis $g_{1},\cdots ,g_{m}$ with the relation that
$$g_{1}\cdots g_{m}=1.$$
One can see Section 8 of \cite{DK} for the geometric meaning of each generator $g_{i}$.

Choose a set of rational numbers $\mu=(\mu_{1},\cdots,\mu_{m})$ satisfying
\begin{align*}
& 0< \mu_{i}< 1, 1\le {i}\le m;\\
& | \mu |:\,=\sum_{i=1}^{m}\mu_{i}\in \mathbb{Z}.
\end{align*}
Define $\mathcal{L}_{\mu}$ to be the local system on $U$ by the homomorphism
$$\chi :\, \pi_{1}(U)\to \C^{*}, \, g_{i}\mapsto e^{-2\pi \i\mu_{i}}.$$
Since $|\mu|\in \mathbb{Z}$, the above homomorphism is well-defined. We are interested in the cohomology $H^{*}(U,\mathcal{L}_{\mu})$ with Hodge decomposition defined as follows.

Let $d$ be the least common denominator of $\mu_{1},\cdots,\mu_{m}$.
We define  $X$ to be the smooth projective variety of $\mathbb{P}^{m-1}$ by the equations
\begin{align}\label{variety of hyperplane}
a_{1j}z_{1}^{d}+\cdots+a_{mj}z_{m}^{d}=0, 0\le j\le n,
\end{align}
where the coefficients $a_{ij}$, $1\le i\le m,0\le j\le n$ determine the hyperplanes as \eqref{linear form}, and $[z_{1},\cdots, z_{m}]$ is the homogeneous coordinate of $\mathbb{P}^{m-1}$.

Define the finite group $G$ to be
$$G=\pi_{1}(U)/\pi_{1}(U)^{d},$$
which is isomorphic to the additive group $$(\mathbb{Z}/d)^{m}/<\sum_{i=1}^{m}e_{i}>,$$ where $e_{i}$ is the generator of the $i$-th component of $(\mathbb{Z}/d)^{m}$. The
group $G$ acts on $X$ as an automorphism induced by the well-defined action on $\mathbb{P}^{m-1}$:
$$[z_{1},\cdots, z_{m}] \mapsto [\tilde{g}_{1}z_{1},\cdots,\tilde{g}_{m}z_{m}],$$
where $\tilde{g}_{i}$ is the image of $g_{i}$ in $G$, $1\le i\le m$.
The action of $G$ on $X$ also induces an action on $H^{*}(X,\mathbb{Z})$.
By the definition of $d$ and $\chi:\, \pi_{1}(U)\to \C^{*}$, $\chi$ can also be considered as a character in $\text{Hom}(G,\C^{*})$.
Let $H^{n}_{\chi}(X,\C)$ be defined in the construction of eigenperiod map, and $$H^{n}_{\chi}(X,\C)=\bigoplus_{p+q=n} H^{p,q}_{\chi}(X)$$ be the decomposition as in \eqref{chi decomp}.

Lemma 8.1 of \cite{DK} implies that $$H^{i}(U,\mathcal{L}_{\mu})=0\text{ if }i\ne n, \text{ and } H^{n}(U,\mathcal{L}_{\mu})\simeq H^{n}_{\chi}(X,\C).$$
Then the isomorphism $H^{n}(U,\mathcal{L}_{\mu})\simeq H^{n}_{\chi}(X,\C)$ gives a Hodge structure on $H^{n}(U,\mathcal{L}_{\mu})$ by defining $$H^{n}(U,\mathcal{L}_{\mu})^{p,q}=H^{p,q}_{\chi}(X),\text{ for any }p+q=n.$$

Denote $h^{p,q}_{\chi}(X)=\dim_{\C}H^{p,q}_{\chi}(X)$. Then Lemma 8.2 of \cite{DK} implies that
$$h^{p,q}_{\chi}(X)={|\mu|-1 \choose p}{m-1-|\mu| \choose q},$$
and
$$\dim_{\C}H^{n}_{\chi}(X,\C)={m-2 \choose n}.$$

Now we consider the case that $|\mu|=\mu_{1}+\cdots +\mu_{m}=n+1$. Then by direct computations
\begin{eqnarray*}
h^{n,0}_{\chi}(X)&=&{n \choose n}{m-n-2 \choose 0}=1,\\
h^{n-1,1}_{\chi}(X)&=&{n \choose n-1}{m-n-2 \choose 1}=n(m-n-2).
\end{eqnarray*}
We know that the moduli space $\mathcal{P}_{m,n}$ of the ordered sets of $m$ hyperplanes in general linear position in $\mathbb{P}^{n}$ is a quasi-projective algebraic variety of dimension $n(m-n-2)$. 

From \cite{DK}, we can define the eigenperiod map $$\Phi_{\chi}:\, \tilde{\mathcal{P}}_{m,n} \to D_{\chi}$$ as before on the universal covering space
$\tilde{\mathcal{P}}_{m,n}$ of $\mathcal{P}_{m,n}$. Moreover from Theorem 8.3 of \cite{DK}, \cite{Varchenko} and \cite{JL}, we know that the eigenperiod map $\Phi_{\chi}:\,
\tilde{\mathcal{P}}_{m,n} \to D_{\chi}$ is a local isomorphism, if $|\mu|=n+1$. Hence the Hodge subbundle $\mathcal{H}$ of Definition \ref{local-Torelli} exists and is
$$\mathcal{H}=\text{Hom}( \mathscr{F}_{\chi}^{n},  \mathscr{F}_{\chi}^{n-1}/ \mathscr{F}_{\chi}^{n}).$$

When $n=1$, hyperplanes in $\mathbb{P}^{1}$ are points in $\mathbb{P}^{1}$, and the arrangement of $m$ hyperplanes in $\mathbb{P}^{1}$ with local system $\mathcal{L}_{\mu}$, $|\mu|=2$ is equivalent to the hypergeometric form
$$\omega_{\mu}=(z-z_{1})^{-\mu_{1}}\cdots(z-z_{m-1})^{-\mu_{m-1}}dz.$$
This is Deligne-Mostow theory studied in their famous paper \cite{DM}.

\end{example}


\section{Extensions of moduli spaces}\label{compactifications}
In this section, we study three completions, or extensions, of the moduli spaces and prove the equivalence of them.
Then we prove that the extended period maps over the extended moduli spaces still satisfy the Griffiths transversality. Most of the results in this section are either review or extensions of the results from \cite{LS}. For reader's convenience we give proofs of some of the key results.

\subsection{Extensions of moduli spaces}\label{compactification}
In this section we assume that $\Z$ is quasi-projective. By Hironaka's resolution of singularity theorem, $\Z$ admits a compactification $\bar{\mathcal Z}_m$ such that
$\bar{\mathcal Z}_m$ is a smooth projective variety, and that $\bar{\mathcal Z}_m\setminus \Z$ is a divisor with simple normal crossings.

Let $\Z'\supseteq \Z$ be the maximal subset of $\bar{\mathcal Z}_m$ to which the period map $\Phi_{\Z}:\, \Z\to D/\Gamma$ extends continuously and let $$\Phi_{\Z'} : \, \Z'\to D/\Gamma$$ be the extended map. Then one has the commutative diagram
\begin{equation*}
\xymatrix{
\Z \ar@(ur,ul)[r]+<16mm,4mm>^-{\Phi_{\Z}}\ar[r]^-{i} &\Z' \ar[r]^-{\Phi_{\Z'}} & D/\Gamma.
}
\end{equation*}
with $i :\, \Z\to \Z'$ the inclusion map.

The following result is a simple corollary of Griffiths \cite{Griffiths3}.

\begin{proposition}\label{openness1}
The space $\Z'$ is a Zariski open subset of $\bar{\mathcal Z}_m$ with $$\text{codim}_{\C}(\bar{\mathcal Z}_m\setminus \Z')\ge 1,$$ and $\Z'\setminus \Z$ consists of the points of $\bar{\mathcal Z}_m\setminus \Z$ around which the Picard-Lefschetz transformations are trivial. Moreover the extended period map $$\Phi_{\Z'}:\, \Z'\to D/\Gamma$$ is a proper holomorphic map.
\end{proposition}
\begin{proof}
Since $\bar{\mathcal Z}_m\setminus \Z$ is a divisor with simple normal crossings, for any point $p$ in $\bar{\mathcal Z}_m\setminus \Z$ we can find a neighborhood $U$ of $p$ in $\bar{\mathcal Z}_m$, which is isomorphic to a polycylinder $\Delta^n$, such that
$$U\cap \Z\simeq (\Delta^*)^k\times \Delta^{N-k},$$
where $N$ is the complex dimension of $\bar{\mathcal Z}_m$ and $0\le k \le N$.

Let $T_i$, $1\le i\le k$ be the image of the $i$-th fundamental group of $(\Delta^*)^k$ under the representation of monodromy. Then the $T_i$'s are called the Picard-Lefschetz
transformations.

Let us define the subspace $\Z''\subset \bar{\mathcal{Z}}_{m}$ which contains $\Z$ and the points in $\bar{\mathcal{Z}}_{m}\setminus \Z$ around which the Picard-Lefschetz transformations are of finite order, hence trivial by Lemma \ref{trivial monodromy}.  Now we claim that $\Z'=\Z''$.

From Theorem 9.6 in \cite{Griffiths3} and its proof, or Corollary 13.4.6 in \cite{CMP}, we know that $\Z''$ is open and dense in $\bar{\mathcal{Z}}_{m}$ and the period map $\Phi_{\Z}$ extends to a holomorphic map $$\Phi_{\Z''} :\, \Z'' \to D/\Gamma$$ which is proper.

In fact, as proved in Theorem 3.1 of \cite{sz}, which follows directly from Propositions 9.10 and 9.11 of \cite{Griffiths3}, $\bar{\mathcal{Z}}_{m}\backslash \Z''$ consists of the components of divisors in $\bar{\mathcal{Z}}_{m}$ whose Picard-Lefschetz transformations are of infinite order, and therefore
$\Z''$ is a Zariski open subset in $\bar{\mathcal{Z}}_{m}$.
Hence by the definition of $\Z'$, we know that $\Z''\subseteq \Z'$.

Conversely, let $q$ be any point in $\Z'$ with image $u=\Phi_{\Z'}(q)\in D/\Gamma$. By the definition of $\Z'$, we can choose the points $$q_{k}\in \Z, \ k=1,2,\cdots$$ such that $q_{k}\longrightarrow q$ with images $u_{k}=\Phi_{\Z}(q_{k})\longrightarrow u$ as $k\longrightarrow \infty$. Since $$\Phi_{\Z''}:\, \Z''\to D/\Gamma$$ is proper, the sequence $$\{q_{k}\}_{k=1}^{\infty}\subset (\Phi_{\Z''})^{-1}(\{u_{k}\}_{k=1}^{\infty})$$ has the limit point $q$ in $\Z''$, that is to say $q\in \Z''$ and $\Z'\subseteq \Z''$.

Therefore we have proved that $\Z'=\Z''$ and $\Z'\setminus \Z$ consists of the points around which the Picard-Lefschetz transformations are trivial.

Griffiths \cite{Griffiths3} proved that the extended period map $\Phi_{\Z'}:\, \Z'\to D/\Gamma$ is a proper holomorphic map. See also \cite{Sommese} in which $\Phi_{\Z'}$ is called the Griffiths extension.
\end{proof}


\subsection{Hodge metric completion}\label{completion}
In Section \ref{compactification} we have discussed the Griffiths extension $$\Phi_{\Z'}:\, \Z'\to D/\Gamma.$$ In this section, we use the local Torelli theorem to give a geometric interpretation of $\Z'$, which is the Hodge metric completion of $\Z$.

In \cite{GS}, Griffiths and Schmid studied the {Hodge metric} on the period domain $D$, which we denote by $h$. The Hodge metric is a complete homogeneous metric. In the case that local Torelli holds, the tangent maps of the period map $\Phi_{\Z}$ and the lifted period map $\Phi$ both are injective. It follows that the pull-backs of $h$ by $\Phi_{\Z}$ and $\Phi$ on $\Z$ and $\T$ respectively are well-defined K\"ahler metrics. For convenience we will still call these pull-back metrics the Hodge metrics.

Let us denote by $\ZZ$ the completion of $\Z$ in $\bar{\mathcal Z}_m$ with respect to the Hodge metric. Then $\ZZ$ is the smallest complete space with respect to the Hodge metric that contains $\Z$.

Now we recall some basic properties about metric completion space we are using in this paper. First we know that the metric completion space of a connected space is still connected. Therefore, $\ZZ$ is connected.

Suppose $(X, d)$ is a metric space with the metric $d$. Then the metric completion space of $(X, d)$ is unique in the following sense: if $\bar{X}_1$ and $\bar{X}_2$ are complete metric spaces that both contain $X$ as a dense set, then there exists an isometry $$f: \bar{X}_1\rightarrow \bar{X}_2$$ such that $f|_{X}$ is the identity map on $X$. Moreover, the metric completion space $\bar{X}$ of $X$ is the smallest complete metric space containing $X$ in the sense that any other complete space that contains $X$ as a subspace must also contains $\bar{X}$ as a subspace.
Hence the Hodge metric completion space $\ZZ$ is unique up to isometry, although the compact space $\bar{\mathcal{Z}}_{m}$ may not be unique. This means that our definition of $\ZZ$ is intrinsic.

Moreover, suppose $\bar{X}$ is the metric completion space of the metric space $(X, d)$. If there is a continuous map $f: \,X\rightarrow Y$ which is a local isometry with $Y$ a complete space, then there exists a continuous extension $\bar{f}:\,\bar{X}\rightarrow{Y}$ such that $\bar{f}|_{X}=f$.
Since $D/\Gamma$ with the Hodge metric $h$ is complete, we can extend the period map to a continuous map $$\Phi_{\ZZ}:\, \ZZ\to D/\Gamma.$$

With the above preparations, we can prove the following proposition.

\begin{proposition}\label{openness2} We have
$\Z'=\ZZ$, and the extended period map $$\Phi_{\ZZ}:\, \ZZ\to D/\Gamma$$ is proper and holomorphic.
\end{proposition}
\begin{proof}
Since we have the continuously extended  map $$\Phi_{\ZZ}:\, \ZZ\to D/\Gamma,$$ we see that $\ZZ\subseteq \Z'$ by the definition of $\Z'$.

Conversely, let $q$ be any point in $\Z'\setminus \Z$.
Fix a point $p$ in $\Z$.
Since $q$ is mapped into $D/\Gamma$ via the extended period map $\Phi_{\Z'}$, we get that $q$ has  finite Hodge distance from $p$. Here the Hodge distance is defined by the pull-back Hodge metric on $\Z$. Therefore $q$ lies in the Hodge metric completion $\ZZ$ of $\Z$. So we get that $\Z'\subseteq \ZZ$.

We now have proved that $\Z'=\ZZ$, which, together with Proposition \ref{openness1}, completes the proof of the proposition.
\end{proof}

From now on we will use $\ZZ$ for our discussions, since it has explicit geometric structure which is convenient for our discussion.


\subsection{Properties of the extended period map}
Let $\TT$ be the universal cover of $\ZZ$ with the universal covering map $$\pi_m^H:\,  \TT\to \ZZ.$$ We then have the following commutative diagram
\begin{equation}\label{main-diagram}
\xymatrix{\T \ar[r]^{i_m}\ar[d]^{\pi_m}&\TT\ar[d]^{\pi_m^H}\ar[r]^{\Phi_m^H}&D\ar[d]^{\pi_D}\\
\Z\ar[r]^{i}&\ZZ\ar[r]^{\Phi_{\ZZ}}&D/\Gamma,
}
\end{equation}
where $i_m$ is the lifting of $i\circ \pi_m$ with respect to the covering map $\pi_m^H$ and $\Phi_m^H$ is the lifting of $\Phi_{\ZZ}\circ \pi_m^H$ with respect to the covering map $$\pi_D:\, D\to D/\Gamma.$$

As the lifts of the holomorphic maps $i$ and $\Phi_{\mathcal{Z}_m}^H$ to universal covers, both $i_m$ and ${\Phi}^{H}_{m}$ are easily seen to be holomorphic maps.
There are different choices of $i_m$ and $\Phi_m^H$, but the elementary topological argument as given in Lemma A.1 in the Appendix of \cite{LS} shows that we can choose $i_m$ and $\Phi_m^H$ such that $$\Phi=\Phi_m^H\circ i_m.$$

Let $\T_m\subseteq \TT$ be defined by $$\T_m=i_m(\T).$$
Then we have the following lemma.
\begin{lemma}\label{openness3}
$\T_m=(\pi_m^H)^{-1}(\Z)$, and $i_m: \, \T \to \T_m$ is a covering map.

\end{lemma}
\begin{proof} The proof is an elementary argument in basic topology.
First, from diagram (\ref{main-diagram}), we see that $$\pi_m^H(\T_m)=\pi_m^H(i_m(\T))=i(\pi_m(\T))=\Z,$$ hence $\T_m \subseteq (\pi_m^H)^{-1}(\Z)$.

Conversely, for any $q \in (\pi_m^H)^{-1}(\Z)$, we need to prove that $q\in \T_m$. Let $p=\pi_m^H(q)$. If there exists a $r\in \pi_m^{-1}(q)$ such that $i_m(r)=q$,  then we are
done. Otherwise, since $\TT$ is connected and thus path connected, we can draw a curve $\gamma$ from $i_m(r)$ to $q$ for some $r\in \pi_m^{-1}(q)$. Then we get a circle
$$\Gamma=\pi_m^H(\gamma)$$ in $\ZZ$. But Lemma A.2 in the Appendix of \cite{LS} implies that we can choose $\Gamma$ contained in $\Z$.

Note that $p\in \Gamma$. Since $\pi_m : \, \T \to \Z$ is covering map, we can lift $\Gamma$ to a unique curve $\tilde{\gamma}$ from $r$ to some $r'\in \pi_{m}^{-1}(p)$.  Now
both $\gamma$ and $i_{m}(\tilde{\gamma})$ map to $\Gamma$ via the covering map $$\pi_m^H :\,  \TT \to \ZZ,$$ that is $\gamma$ and and $i_{m}(\tilde{\gamma})$ are both the lifts
of $\Gamma$. By the uniqueness of homotopy lifting, $i_m(r')=q$, i.e. $q\in i_m(\T)=\T_m$.  Therefore we have proved that $$\T_m=(\pi_m^H)^{-1}(\Z),$$ and $\T_m$ is a smooth
complex submanifold of $\TT$.

To show that $i_m$ is a covering map, note that for any small enough open neighborhood $U$ in $\T_{m}$, the restricted map $$\pi_m^H|_{U} :\, U \to V=\pi_m^H(U)\subset \Z$$ is biholomorphic, and there exists a disjoint union $\cup_{i}V_{i}$ of open subsets in $\T$ such that $$\cup_{i}V_{i}=(\pi_{m})^{-1}(V)$$ and $\pi_{m}|_{V_{i}}:\, V_{i}\to V$ is biholomorphic. Then from the commutativity of diagram \eqref{main-diagram}, we have that $\cup_{i}V_{i}=(i_{m})^{-1}(U)$ and $$i_{m}|_{V_{i}}:\, V_{i}\to U$$ is biholomorphic. Therefore $i_m:\, \T \to \T_m$ is also a covering map.
\end{proof}

Lemma \ref{openness3} also implies that $\T_m$ is an open dense complex submanifold of $\TT$ and $\TT\setminus \T_m$ is an analytic subvariety of $\TT$ with $\text{codim}_\C(\TT\setminus \T_m)\ge 1$.

Indeed, from Proposition \ref{openness1} and Proposition \ref{openness2}, we have that $\ZZ\setminus\Z$ is an analytic subvariety of $\ZZ$. On the other hand, from Lemma \ref{openness3}, we know that  $\TT\setminus \T_m$ is the inverse image of $\ZZ\setminus \Z$ under the covering map $$\pi^{H}_m:\, \TT \to \ZZ.$$ This implies that $\TT\setminus \T_m$ is an analytic subvariety of $\TT$.

The identification $\T_m=(\pi_m^H)^{-1}(\Z)$ implies that $\T_m$ is a covering of $\Z$. Since $\T$ is the universal cover, we have proved that $\T$ is also the universal cover of $\T_m$ for each $m$.

\begin{lemma}\label{extendedtransversality}
The extended holomorphic map $\Phi_m^H :\, \TT \to D$ satisfies the Griffiths
transversality.
\end{lemma}
\begin{proof}
Let $\mathrm{T}^{1,0}_hD$ be the horizontal subbundle.
Since $\Phi_m^H : \TT \to D$ is a holomorphic map, the tangent map
$$d\Phi_m^H : \, \mathrm{T}^{1,0}\TT \to \mathrm{T}^{1,0}D$$
is at least continuous. We only need to show that the image of $d\Phi_m^H$ is contained in the horizontal tangent bundle $\mathrm{T}^{1,0}_hD$.

Since $\mathrm{T}^{1,0}_hD$ is closed in $\mathrm{T}^{1,0}D$, the inverse image $(d\Phi_m^H)^{-1}(\mathrm{T}^{1,0}_hD)$ is a closed subset in $\mathrm{T}^{1,0}\TT$. But
$\Phi_m^H|_{\T_m}$ satisfies the Griffiths transversality, i.e. $$(d\Phi_m^H)^{-1}(\mathrm{T}^{1,0}_hD)$$ contains $\mathrm{T}^{1,0}\T_m$, which is open in $\mathrm{T}^{1,0}\TT$.
Hence $(d\Phi_m^H)^{-1}(\mathrm{T}^{1,0}_hD)$ contains the closure of $\mathrm{T}^{1,0}\T_m$, which is $\mathrm{T}^{1,0}\TT$.
\end{proof}


\section{Boundedness of period maps and affine structures}\label{bounded periods}

In this section we review the definitions and basic properties of period domains from Lie theory point of views. We consider the nilpotent Lie subalgebra $\mathfrak{n}_+$ of $\mathfrak{g}$ and define the corresponding unipotent group to be $$N_+=\exp(\mathfrak{n}_+).$$
Then we recall some results from our previous paper \cite{LS}, including the boundedness of period maps which asserts, as given in Theorem \ref{boundedness}, that the image of the period map $\Phi :\, \T \to D$ lies in $N_{+}\cap D$ as a bounded subset, and the result that there exists a global affine structure on $\T$ which is defined by the holomorphic map in \eqref{maptoA}, see Theorem \ref{affine-T}.
\\

First we fix a point $p$ in $\T$ and its image $o=\Phi(p)$ as the reference points or base points. Let us introduce the notion of adapted basis for the given Hodge decomposition or Hodge filtration.
For the fixed point $p\in \T$ and $f^k=\dim F^k_p$ for any $0\leq k\leq n$, we call a basis
\begin{equation*}
\xi=\left\{ \xi_0, \cdots, \xi_{f^{n}-1},\xi_{f^{n}}, \cdots ,\xi_{f^{n-1}-1} \cdots, \xi_{f^{k+1}}, \cdots, \xi_{f^k-1}, \cdots, \xi_{f^{1}},\cdots , \xi_{f^{0}-1} \right\}
\end{equation*}
 of $H^n_{pr}(\Y _p, \mathbb{C})$ an {adapted basis for the given Hodge decomposition}
$$H^n_{pr}(\Y _p, {\mathbb{C}})=H^{n, 0}_p\oplus H^{n-1, 1}_p\oplus\cdots \oplus H^{1, n-1}_p\oplus H^{0, n}_p,$$
if it satisfies $
H^{k, n-k}_p=\text{Span}_{\mathbb{C}}\left\{\xi_{f^{k+1}}, \cdots, \xi_{f^k-1}\right\}$ with $h^{k,n-k}=f^k-f^{k+1}$.

We call a basis
\begin{align*}
\zeta=\left\{ \zeta_0, \cdots, \zeta_{f^{n}-1},\zeta_{f^{n}}, \cdots ,\zeta_{f^{n-1}-1} \cdots, \zeta_{f^{k+1}}, \cdots, \zeta_{f^k-1}, \cdots,\zeta_{f^{1}},\cdots , \zeta_{f^{0}-1} \right\}
\end{align*}
of $H^n_{pr}(\Y _p, {\mathbb{C}})$ an {adapted basis for the given filtration}
\begin{align*}
F^{n}_p\subseteq F^{n-1}_p\subseteq\cdots\subseteq F^0_p
\end{align*}
if it satisfies $F^{k}_p=\text{Span}_{\mathbb{C}}\{\zeta_0, \cdots, \zeta_{f^k-1}\}$ with $\text{dim}_{\mathbb{C}}F^{k}=f^k$.
For the convenience of notations, we set $f^{n+1}=0$ and $m=f^0$.

The blocks of an $m\times m$ matrix $T$ are set as follows:
for each $0\leq \alpha, \beta\leq n$, the $(\alpha, \beta)$-th block $T^{\alpha, \beta}$ is
\begin{align}\label{block}
T^{\alpha, \beta}=\left(T_{ij}\right)_{f^{-\alpha+n+1}\leq i \leq f^{-\alpha+n}-1, \ f^{-\beta+n+1}\leq j\leq f^{-\beta+n}-1},
\end{align} where $T_{ij}$ is the entries of
the matrix $T$. In particular, $T =[T^{\alpha,\beta}]_{0\le \alpha,\beta \le n}$ is called a {block lower triangular matrix} if
$T^{\alpha,\beta}=0$ whenever $\alpha<\beta$.

Let $H_{\mathbb{F}}=H^n_{pr}(X, \mathbb{F})$, where $\mathbb{F}$ can be chosen as $\mathbb{Z}$, $\mathbb{R}$, $\mathbb{C}$. We define the complex Lie group
\begin{align*}
G_{\mathbb{C}}=\{ g\in GL(H_{\mathbb{C}})|~ Q(gu, gv)=Q(u, v) \text{ for all } u, v\in H_{\mathbb{C}}\},
\end{align*}
and the real one
\begin{align*}
G_{\mathbb{R}}=\{ g\in GL(H_{\mathbb{R}})|~ Q(gu, gv)=Q(u, v) \text{ for all } u, v\in H_{\mathbb{R}}\}.
\end{align*}
Griffiths in \cite{Griffiths1} showed that $G_{\mathbb{C}}$ acts on $\check{D}$ transitively and so does $G_{\mathbb{R}}$ on $D$. The stabilizer of $G_{\mathbb{C}}$ on $\check{D}$ at the fixed point $o$ is
$$B=\{g\in G_\C| gF_p^k=F_p^k,\ 0\le k\le n\},$$
and the one of $G_{\mathbb{R}}$ on $D$ is $V=B\cap G_\mathbb{R}$.
Thus we can realize $\check{D}$ as
$$\check{D}=G_\C/B,\text{ and }D=G_\mathbb{R}/V$$
so that $\check{D}$ is an algebraic manifold and $D\subseteq
\check{D}$ is an open complex submanifold.

The Lie algebra $\mathfrak{g}$ of the complex Lie group $G_{\mathbb{C}}$ is
\begin{align*}
\mathfrak{g}&=\{X\in \text{End}(H_\mathbb{C})|~ Q(Xu, v)+Q(u, Xv)=0, \text{ for all } u, v\in H_\mathbb{C}\},
\end{align*}
and the real subalgebra
$$\mathfrak{g}_0=\{X\in \mathfrak{g}|~ XH_{\mathbb{R}}\subseteq H_\mathbb{R}\}$$
is the Lie algebra of $G_\mathbb{R}$.
Note that $\mathfrak{g}$ is a simple complex Lie algebra and contains $\mathfrak{g}_0$ as a real form, i.e. $\mathfrak{g}=\mathfrak{g}_0\oplus \i\mathfrak{g}_0$.
Let us denote the complex conjugation of $\mathfrak{g}$ with respect to  the real form $\mathfrak{g}_0$ by $\tau_0$.

On the Lie algebra $\mathfrak{g}$ we can give a Hodge structure of weight zero by
\begin{align*}
\mathfrak{g}=\bigoplus_{k\in \mathbb{Z}} \mathfrak{g}^{k, -k}\quad\text{with}\quad\mathfrak{g}^{k, -k}=\{X\in \mathfrak{g}|XH^{r, n-r}_p\subseteq H^{r+k, n-r-k}_p,\ \forall r \}.
\end{align*}

By definition of $B$ the Lie algebra $\mathfrak{b}$ of $B$ has the form $\mathfrak{b}=\bigoplus_{k\geq 0} \mathfrak{g}^{k, -k}$. Then the Lie algebra
$\mathfrak{v}_0$ of $V$ is
$$\mathfrak{v}_0=\mathfrak{g}_0\cap \mathfrak{b}=\mathfrak{g}_0\cap \mathfrak{b}\cap\bar{\mathfrak{b}}=\mathfrak{g}_0\cap \mathfrak{g}^{0, 0}.$$ With the above isomorphisms, the holomorphic tangent space of $\check{D}$ at the base point is naturally isomorphic to $\mathfrak{g}/\mathfrak{b}$.

Let us consider the nilpotent Lie subalgebra
$\mathfrak{n}_+:=\oplus_{k\geq 1}\mathfrak{g}^{-k,k}$.  Then one
gets the isomorphism $\mathfrak{g}/\mathfrak{b}\cong
\mathfrak{n}_+$. We denote the corresponding unipotent Lie group to be
$$N_+=\exp(\mathfrak{n}_+).$$

As $\text{Ad}(g)(\mathfrak{g}^{k, -k})$ is in  $\bigoplus_{i\geq
k}\mathfrak{g}^{i, -i} \text{ for each } g\in B,$ the subspace
$\mathfrak{b}\oplus \mathfrak{g}^{-1, 1}/\mathfrak{b}\subseteq
\mathfrak{g}/\mathfrak{b}$ defines an Ad$(B)$-invariant subspace. By
left translation via $G_{\mathbb{C}}$,
$\mathfrak{b}\oplus\mathfrak{g}^{-1,1}/\mathfrak{b}$ gives rise to a
$G_{\mathbb{C}}$-invariant holomorphic subbundle of the holomorphic
tangent bundle. It will be denoted by $\mathrm{T}^{1,0}_{h}\check{D}$,
and will be referred to as the horizontal tangent subbundle. One can
check that this construction does not depend on the choice of the
base point.

The horizontal tangent subbundle, restricted to $D$, determines a subbundle $\mathrm{T}_{h}^{1, 0}D$ of the holomorphic tangent bundle $\mathrm{T}^{1, 0}D$ of $D$.
The $G_{\mathbb{C}}$-invariance of $\mathrm{T}^{1, 0}_{h}\check{D}$ implies the $G_{\mathbb{R}}$-invariance of $\mathrm{T}^{1, 0}_{h}D$. Note that the horizontal
tangent subbundle $\mathrm{T}_{h}^{1, 0}D$ can also be constructed as the associated bundle of the principle bundle $V\to G_\mathbb{R} \to D$ with the adjoint
representation of $V$ on the space $\mathfrak{b}\oplus\mathfrak{g}^{-1,1}/\mathfrak{b}$.

Let $\mathscr{F}^{k}$, $0\le k \le n$ be the  Hodge bundles on $D$
with fibers $\mathscr{F}^{k}_{s}=F_{s}^{k}$ for any $s\in D$. As another
interpretation of the horizontal bundle in terms of the Hodge
bundles $\mathscr{F}^{k}\to \check{D}$, $0\le k \le n$, one has
\begin{align}\label{horizontal}
\mathrm{T}^{1, 0}_{h}\check{D}\simeq \mathrm{T}^{1, 0}\check{D}\cap \bigoplus_{k=1}^{n}\text{Hom}(\mathscr{F}^{k}/\mathscr{F}^{k+1}, \mathscr{F}^{k-1}/\mathscr{F}^{k}).
\end{align}

We remark that the elements in $N_+$ can be realized as nonsingular block lower triangular matrices with identity blocks in the diagonal; elements in $B$ can be realized as nonsingular block upper triangular matrices.
If $c, c'\in N_+$ such that $cB=c'B$ in $\check{D}$, then $$c'^{-1}c\in N_+\cap B=\{I \},$$ i.e. $c=c'$. This means that the matrix representation in $N_+$ of the unipotent orbit $N_+(o)$ is unique. Therefore with the fixed base point $o\in \check{D}$, we can identify $N_+$ with its unipotent orbit $N_+(o)$ in $\check{D}$ by identifying an element $c\in N_+$ with $[c]=cB$ in $\check{D}$. Then $$N_+\subseteq\check{D}$$ is meaningful. In particular, when the base point $o$ is in $D$, we have $N_+\cap D\subseteq D$.

We define a submanifold of the Teichm\"uller space
\begin{align*}
\check{\T}=\Phi^{-1}(N_+\cap D).
\end{align*}
In Proposition 1.3 of \cite{LS}, we showed that $\mathcal{T}\backslash\check{\mathcal{T}}$ is an analytic subvariety of $\mathcal{T}$ with $\text{codim}_{\mathbb{C}}(\mathcal{T}\backslash\check{\mathcal{T}})\geq 1$.

At the base point $$\Phi(p)=o\in N_+\cap D,$$ we have identifications of the tangent spaces
$$\mathrm{T}_{o}^{1,0}N_+=\mathrm{T}_o^{1,0}D\simeq \mathfrak{n}_+,$$ and the exponential
map $$\text{exp} :\, \mathfrak{n}_+ \to N_+$$ is an isomorphism.

The Hodge metric on $\mathrm{T}_o^{1,0}D$, which is the restriction of the natural homogeneous metric on $D$ at the base point, induces an Euclidean metric on $N_+$ so that $\text{exp} :\, \mathfrak{n}_+ \to N_+$ is an isometry.

In Theorem 4.1 of \cite{LS}, we prove the following theorem, which was raised by Griffiths as Conjecture 10.1 in his paper \cite{Griffiths4}.

\begin{theorem}\label{boundedness}
The image of $\Phi : \T\to D$ lies in $N_+\cap D$ and is bounded with respect to the Euclidean metric on $N_+$.
\end{theorem}

We remark that Theorem \ref{boundedness} is proved in \cite{LS} for general period maps and analytic families. To be precise, let $f:\, \X\to S$ be any analytic family of polarized manifolds over a quasi-projective manifold $S$, and let $\tilde{S}$ be the universal covering space of $S$ with the universal covering map $\pi:\, \tilde{S}\to S$. We can define the period map $$\Phi:\, S\to D/\Gamma$$ with the monodromy representation $\rho:\, \pi_{1}(S)\to \Gamma$, and the lifted period map $$\tilde{\Phi}:\, \tilde{S}\to D$$ such that the following diagram commutes
$$\xymatrix{
\tilde{S} \ar[r]^-{\tilde{\Phi}} \ar[d]^-{\pi} & D\ar[d]^-{\pi_{D}}\\
S \ar[r]^-{\Phi} & D/\Gamma .}$$
Then the conclusion of Theorem \ref{boundedness} still holds for the lifted period map $\tilde{\Phi}:\, \tilde{S}\to D$.

Recall that we have defined the extended period map $$\Phi_m^H : \TT \rightarrow D$$ in diagram \eqref{main-diagram}.
Then the extended period map $\Phi_m^H$ is still bounded by Corollary 6.1 in \cite{LS}, which follows directly from  Theorem \ref{boundedness} and the Riemann extension theorem.

\begin{corollary}\label{extended-boundedness}
The image of the extended period map $\Phi_m^H : \TT \rightarrow D$ lies in $N_+\cap D$ and is bounded with respect to the Euclidean metric on $N_+$.
\end{corollary}

Now we review the definition of complex affine structure on a complex manifold.
\begin{definition}
Let $M$ be a complex manifold of complex dimension $n$. If there
is a coordinate cover $\{(U_i,\,\phi_i);\, i\in I\}$ of M such
that $\phi_{ik}=\phi_i\circ\phi_k^{-1}$ is a holomorphic affine
transformation on $\mathbb{C}^n$ whenever $U_i\cap U_k$ is not
empty, then $\{(U_i,\,\phi_i);\, i\in I\}$ is called a complex
affine coordinate cover on $M$ and it defines a holomorphic affine structure on $M$.
\end{definition}

Let us consider $$\a=d\P_{p}(\mathrm{T}_{p}^{1,0}\T)\subseteq \mathrm{T}_o^{1,0}D\simeq \mathfrak{n}_+$$ where $p$ is the base point in $\T$ with $\tilde{\Phi}(p)=o$.
By Griffiths transversality, $\a\subseteq \g^{-1,1}$ is an abelian subspace, therefore $\a \subseteq \mathfrak{n}_+$ is an abelian subalgebra of $\mathfrak{n}_+$
determined by the tangent map of the period map
$$d\P :\, \mathrm{T}^{1,0}{\T} \to \mathrm{T}^{1,0}D.$$
Consider the corresponding Lie group
$$A\triangleq \exp(\a)\subseteq N_+.$$
Then $A$ can be considered as a complex Euclidean subspace of $N_{+}$ with the induced Euclidean metric from $N_+$.

  Define the projection map $P:\, N_+\cap D\to A \cap D$ by
$$ P=\text{exp}\circ p\circ \text{exp}^{-1}$$
where $\text{exp}^{-1}:\, N_+ \to \mathfrak{n}_+$ is the inverse of the isometry $\text{exp}:\, \mathfrak{n}_+ \to N_+$, and $$p:\, \mathfrak{n}_+\to \a$$ is the projection map from the complex Euclidean space $\mathfrak{n}_+$  to its  Euclidean subspace $\mathfrak{a}$.

The period map $\Phi : \, \T\to N_+\cap D$ composed
with the projection map $P$ gives a holomorphic map
\begin{equation}\label{maptoA}
\Psi :\,  \T \to A\cap D,
\end{equation}
that is $\Psi=P \circ \Phi$.

Moreover in Theorem 6.4 of \cite{LS}, we proved under the condition of strong local Torelli that the holomorphic map $\Psi$ is bounded and defines a global affine structure on the Teichm\"uller space $\T$.
\begin{theorem}\label{affine-T}
The holomorphic map $\Psi$ defines a global affine structure on the Teichm\"uller space $\T$ of polarized manifolds containing $(X,L)$, provided that the strong local Torelli holds for the polarized manifold $(X,L)$ in the T-class.
\end{theorem}
 This theorem follows from the boundedness of the period map $\Phi$, which gives the boundedness of $\Psi$, and strong local Torelli for $\Phi$ which gives the nondegeneracy of $\Psi$, so that we can pull back the affine structure of $A \simeq\mathbb{C}^N$.

\section{Strong local Torelli implies global Torelli}\label{global Torelli}
In this section, we prove our main theorem in this paper, which asserts that  strong local Torelli implies global Torelli on the Torelli space for polarized manifolds in the T-class.

First, we prove the existence of the affine structure on $\TT$ induced by the extended map $$\Psi_{m}^{H}:\, \TT\to A\cap D$$ which follows from the proof that $\Psi_m^H$ is nondegenerate. Then we apply a lemma of Griffiths and Wolf to conclude that the holomorphic map $\Psi_{m}^{H}$ is a covering map. Furthermore, by using the affine structures we show that $\Psi_{m}^{H}$ is in fact a biholomorphic map, which implies that $\TT$ does not depend on $m$.

With this we prove that the image $\T_{m}$ of $i_{m}:\, \T\to \TT$ is identical to the Torelli space $\T'$.  Here the level structures again play important roles.
Moreover, we prove that the extended period map $$\Phi_{m}^{H}:\, \TT\to D$$ is an embedding.
At last we arrive at our theorem asserting that the period map $$\Phi':\,\T'\to N_{+}\cap D$$ is injective.

Note that in this paper, unless otherwise specified, we only consider the polarized manifold $(X,L)$, which belongs to the T-class and satisfies strong local Torelli.

Let $\Z$ be defined as in Definition \ref{T-class} and $\ZZ$, $\T$, $\T'$ and $\TT$ be defined as before. First recall diagram \eqref{main-diagram},
\begin{equation}
\xymatrix{\T \ar[r]^{i_m}\ar[d]^{\pi_m}&\TT\ar[d]^{\pi_m^H}\ar[r]^{\Phi_m^H}&D\ar[d]^{\pi_D}\\
\Z\ar[r]^{i}&\ZZ\ar[r]^{\Phi_{\ZZ}}&D/\Gamma.
}
\end{equation}

From Corollary \ref{extended-boundedness},
we can define the holomorphic map $$\Psi_m^H:\, \TT\to A\cap D$$ by
the extended period map $\Phi_m^H:\,\TT\to N_+\cap D$ composed with the projection map $$P:\, N_+\cap D\to A\cap D.$$

Recall that we have defined $$\T_m=i_m(\T)\subset \T_m^H$$ in Section \ref{compactifications}. Lemma \ref{openness3} proves that $\T_{m}\subset \TT$ is a Zariski open
submanifold, and furthermore $i_m:\, \T \to \T_m$ is a covering map, and  $$\T_m=i_m(\T)=(\pi_m^H)^{-1}(\sZ).$$  Therefore we have the holomorphic map $$\Psi_m:\, \T_m\to A\cap
D$$ by restricting $\Psi^H_m$ to $\T_m$,  $$\Psi_m=\Psi_m^H|_{\T_m},$$ or equivalently $\Psi_m=P\circ \Phi_m$.


 We can choose a small neighborhood $U$ of any point in $\T_{m}$ such that
$$\pi_m^H:\, U \to V=\pi_m^H(U)\subset \ZZ,$$
is a biholomorphic map. We can shrink $U$ and $V$ simultaneously such that $\pi_{m}^{-1}(V)=\cup_{\alpha}W_{\alpha}$ and $\pi_{m} :\, W_{\alpha}\to V$ is also
biholomorphic. Choose any $W_{\alpha}$ and denote it by $W=W_{\alpha}$. Then $i_{m} :\, W \to U$ is a biholomorphic map. Since $$\Psi=\Psi_{m}^{H}\circ
i_{m}=\Psi_{m}\circ i_{m},$$ we have $\Psi|_{W}=\Psi_{m}|_{U}\circ i_{m}|_{W}$.

Theorem \ref{affine-T} implies that $\Psi|_{W}$ is biholomorphic onto its
image, if we shrink $W$, $V$ and $U$ again. Therefore $$\Psi_{m}|_{U} :\, U\to A\cap D$$ is biholomorphic onto its image. By pulling back the affine coordinate
chart in $A \simeq \C^N$, we get an induced affine structure on $\T_m$ such that $\Psi_m$ is an affine map.

In conclusion, we have proved the following lemma.

\begin{lemma}\label{affine-Tm}
The holomorphic map $\Psi_m:\, \T_m\to A\cap D$ is a local embedding. In particular, $\Psi_m$ defines a global holomorphic affine structure on $\T_m$.
\end{lemma}

Next we will prove that the affine structure $\Psi_m:\, \T_m\to A\cap D$ can be extended to a global affine structure on $\TT$,  which is equivalent to show that $$\Psi_m^H:\,
\TT\to A\cap D$$ is nondegenerate.

\begin{definition}
Let $M$ be a complex manifold and $N\subset M$ a cloesd subset. Let
$E_0\to M\setminus N$ be a holomorphic vector bundle. Then $E_0$ is
called holomorphically trivial along $N$, if for any point $x\in N$,
there exists an open neighborhood $U$ of  $x$ in $M$ such that
$E_0|_{U\setminus N}$ is holomorphically trivial.
\end{definition}

The following elementary lemma will be needed in the following argument.

\begin{lemma}\label{unique-extension}
The holomorphic vector bundle $E_0\to M\setminus N$ can be extended
to a unique holomorphic vector bundle $E\to M$ such  that
$E|_{M\setminus N}=E_0$, if and only if $E_0$ is holomorphically
trivial along $N$.
\end{lemma}
\begin{proof}
The proof is taken from Proposition 4.4 of \cite{Mckay}. We include it here for the convenience of the readers.

If such an extension $E\to M$ exists uniquely, then we can take a
neighborhood $U$ of $x$ such that $E|_{U}$ is trivial,  and hence
$E_0|_{U\setminus N}$ is trivial.

Conversely, suppose that for any point $x\in N$, there exists an open neighborhood $U_x$ of $x$ in $M$ such  that $E_0|_{U_x\setminus N}$ is trivial. We can extend $E_0$
trivially to each $U_x$. Let us cover $N$ by the open sets $\{U_x:\, x\in N\}$. For any $y\in M\setminus N$, we can choose an open neighborhood $V_y$ of $y$ such that $V_y\cap
N=\emptyset$. Then we can cover $M$ by the open sets in
$$\mathcal{C}=\{U_x:\, x\in N\}\cup \{V_y:\, y\in M\setminus N\}.$$
We denote $\Psi_{V}^{U}$ be the transition function of $E_0$ on
$U\cap V$ if $U\cap V\neq \emptyset$,  where $U,V$ are open subsets
of $M\setminus N$ and $E_0$ is trivial on $U$ and $V$.

If $V_y\cap V_{y'}\neq \emptyset$, then $\Psi_{V_{y'}}^{V_y}$ is
well-defined,  since $V_y\cap V_{y'}\subset M\setminus N$. Then we
define the transition function
$\phi_{V_{y'}}^{V_y}=\Psi_{V_{y'}}^{V_y}$. For the case of $U_x$ and
$V_y$, we can use the same argument to get the transition functions
$\phi_{V_{y}}^{U_x}=\Psi_{V_{y}}^{U_x}$ and
$\phi^{V_{y}}_{U_x}=\Psi^{V_{y}}_{U_x}$.

The remaining case is that $U_x\cap U_{x'}\neq \emptyset$, where $x ,x' \in N$. Then the transition function $\Psi_{U_{x'}\setminus N}^{U_x\setminus N}$ is well-defined and we
define the transition function $\phi_{U_{x'}}^{U_x}=\Psi_{U_{x'}\setminus N}^{U_x\setminus N}$. Clearly the newly defined transition functions $\phi_{V}^U$, for any $U,\, V\in
\mathcal{C}$ satisfy that
$$\phi_{V}^U \circ \phi_{U}^V=\text{id},\text{ and }\phi_{V}^U \circ \phi_{W}^V \circ \phi_{U}^W=\text{id}.$$
Therefore these transition functions determine a unique holomorphic vector bundle $E\to M$ such that $E|_{M\setminus N}=E_0$.
\end{proof}

Let $\mathcal{H}$ be the holomorphic Hodge  subbundle in Definition \ref{local-Torelli},
$$\mathcal{H}\subset \bigoplus_{k=1}^n \text{Hom}(\mathscr{F}^k/\mathscr{F}^{k+1},\mathscr{F}^{k-1}/\mathscr{F}^k),$$
such that $$d\Phi:\, \mathrm{T}^{1,0}\T\stackrel{\sim}{\longrightarrow} \Phi^{*}\mathcal{H}$$ is an isomorphism of holomorphic vector bundles on $\T$.

Let $\mathcal{H}_A$ be the restriction of $\mathcal{H}$ on $A\cap D$. Then by the definition of $A$,  the holomorphic tangent bundle $\mathrm{T}^{1,0}(A\cap D)$ is isomorphic to $\mathcal{H}_A$.

Theorem \ref{affine-T} and Lemma \ref{affine-Tm} give the natural isomorphisms of the holomorphic vector bundles over $\T$ and $\T_m$ respectively
\begin{eqnarray*}
d\Psi:\,\mathrm{T}^{1,0}\T \simeq \Psi^*\mathcal{H}_A,\\
d\Psi_{m}:\,\mathrm{T}^{1,0}\T_m\simeq \Psi_m^*\mathcal{H}_A.
\end{eqnarray*}

We remark that the period map $\Phi_{\ZZ}:\, \ZZ\to D/\Gamma$ can be lifted to the universal cover to get $$\Phi_{m}^{H}: \, \TT\to D,$$ which is due to the fact that around
any point in $\ZZ\setminus \Z$, the monodromy group action is trivial as  proved in Proposition \ref{openness1} and Proposition \ref{openness2}.

In fact the period map $\Phi_{\Z}$ has extension $\Phi_{\ZZ}$ as proved by Griffiths,  in Theorem 9.6 of \cite{Griffiths3}, therefore $\Psi_m:\, \T_m \to A\cap D$ has extension
to
$$\Psi_m^H:\, \TT \to A\cap D$$ which implies that the Hodge bundle $\Psi_m^*\mathcal{H}_A$ has the natural extension $(\Psi_m^H)^*\mathcal{H}_A$
 over $\TT$.

The following lemma is proved by using extension of the period map.

\begin{lemma}\label{Hodge extension}
The isomorphism $\mathrm{T}^{1,0}\T_m\simeq \Psi_m^*\mathcal{H}_A$ of holomorphic vector bundles over $\T_m$ has a unique extension to an isomorphism of holomorphic vector bundles over $\T^H_m$ as
$$\mathrm{T}^{1,0}\TT\simeq (\Psi_m^H)^*\mathcal{H}_A.$$
\end{lemma}
\begin{proof}
For any $p\in \TT\setminus \T_{m}$, we can choose an open neighborhood $U_{p}\subset \TT$ of $p$ such that $\Psi_{m}^{H}(U_{p})$ is contained in an open neighborhood $W\subset A\cap D$ on which $\mathcal{H}_A$ is trivial.

Let $$\bar{o}=\Psi_{m}^{H}(p)\in W \text{ and } H_{A}=\mathcal{H}_A|_{\bar{o}}.$$
Then we have the following trivializations of the Hodge subbundles
$$\mathcal{H}_A|_{W}\simeq W\times H_{A}$$ and
$$(\Psi_m^H)^*\mathcal{H}_A|_{U_{p}}\simeq(\Psi_m^H)^*(\mathcal{H}_A|_{W})|_{U_{p}}\simeq(\Psi_m^H)^*(W\times H_{A})|_{U_{p}}\simeq U_{p}\times H_{A}.$$
Let $U=U_{p}\cap \T_{m}$.
Then the trivialization of the Hodge subbundle on $U$ is
$$\Psi_m^*\mathcal{H}_A|_{U}=((\Psi_m^H)^*\mathcal{H}_A|_{U_{p}})|_{U}\simeq U\times H_{A},$$ which  implies that $\Psi_m^*\mathcal{H}_A$ is holomorphically trivial along $\TT \setminus \T_{m}$.
Therefore by Lemma \ref{unique-extension}, the Hodge subbundle $\Psi_m^*\mathcal{H}_A$ over $\T_m$ has a unique extension over $\TT$, which is $(\Psi_m^H)^*\mathcal{H}_A$ by
continuity.

Note that Lemma \ref{unique-extension} also implies that the holomorphic tangent bundle $\mathrm{T}^{1,0}\T_m$, which is isomorphic to $\Psi_m^*\mathcal{H}_A$, is holomorphically trivial along $\TT\setminus \T_m$.
Hence $\mathrm{T}^{1,0}\T_m$ has a unique extension, which is obviously $\mathrm{T}^{1,0}\TT$. By the uniqueness of such extension we can conclude that
$$\mathrm{T}^{1,0}\TT\simeq (\Psi_m^H)^*\mathcal{H}_A.$$
\end{proof}

With the above preparations, we can prove the following theorem, which is key to our main theorem.

\begin{thm}\label{affine-TmH}
The holomorphic map $\Psi_m^H:\, \TT\to A\cap D$ is nondegenerate. Hence $\Psi_m^H$ defines a global affine structure on $\TT$.
\end{thm}
\begin{proof}
For any point $q\in \TT\setminus \T_m$, we can choose a neighborhood $U_q$ of $q$ in $\TT$ such that $$\mathrm{T}^{1,0}\TT\simeq (\Psi_m^H)^*\mathcal{H}_A$$ is trivial on $U_q$.
Moreover, we can shrink $U_q$ so that $\Psi_m^H(U_q)\subseteq V$, where $V\subset A\cap D$ on which $\mathcal{H}_A$ is trivial. We choose a basis $$\{\Lambda_1, \cdots
,\Lambda_N \}$$ of $\mathcal{H}_A|_{V},$ which is parallel with respect to the natural affine structure on $A\cap D$. Let $$\mu_i=((\Psi_m^H)^*\Lambda_i)|_{U_{q}},\ 1\le i\le
N.$$ Then $\{\mu_1,\cdots,\mu_n\}$ is a basis of $$\mathrm{T}^{1,0}\TT|_{U_{q}}\simeq (\Psi_m^H)^*\mathcal{H}_A|_{U_{q}}.$$

Let $U=U_q\cap \T_m$ and consider a sequence of points  $q_k\in U$ such that $q_k\longrightarrow q$ as $k\longrightarrow \infty$.
Since $\Psi_m:\, \T_m\to A\cap D$ defines a global affine structure on $\T_m$, for any $k\ge1$ we have
$$d\Psi_m(\mu_i|_{q_k})=\sum_{j} A_{ij}(q_k)\Lambda_j|_{o_k},$$
for some non-singular matrix $(A_{ij}(q_k))_{1\le i,j\le N}$, where $o_k=\Psi_m(q_k)$.

Since the affine structure on $\T_m$ is induced from that of $A\cap D$, both the bases
$$\{\mu_1|_{U},\cdots,\mu_N|_{U} \}\text{ and }\{\Lambda_1, \cdots ,\Lambda_N \}$$
are parallel with respect to the affine structures on $\T_m$ and $A\cap D$ respectively.
Hence we have that the matrix  $(A_{ij}(q_k))=(A_{ij})$ is constant matrix.
Therefore
\begin{eqnarray*}
d\Psi^H_m(\mu_i|_{q})&=&\lim_{k\to \infty}d\Psi^H_m(\mu_i|_{q_k})\\
&=&\lim_{k\to \infty}d\Psi_m(\mu_i|_{q_k})\\
&=&\lim_{k\to \infty}\sum_{j} A_{ij}\Lambda_j|_{o_k}\\
&=&\sum_{j} A_{ij}\Lambda_j|_{o_{_\infty}},
\end{eqnarray*}
where $o_\infty =\Psi^H_m(q)$. Since the matrix $(A_{ij})_{1\le i,j\le N}$ is nonsingular, the tangent map $d\Psi^H_m$ is nondegenerate.
\end{proof}

Before moving on, we would like to point out the geometric intuition behind our definition of strong local Torelli. First note that in the proof of Theorem \ref{affine-TmH}, we used substantially the identification $$d\Psi_{m}:\, \mathrm{T}^{1,0}\T_m\simeq \Psi_m^*\mathcal{H}_A$$ on $\T_m$,
which is explicitly given by the contraction map $KS(v)\lrcorner$ at a point $(q,v) \in \mathrm{T}^{1,0}_q \T_m$ given by \eqref{contraction}. Here recall that  $$KS: \, \mathrm{T}^{1,0}_q \T_m \to H^1(X_q, \Theta_{X_q})$$ is the Kodaira-Spencer map for any $q\in \T_m$. From this explicit expression, one can see that this identification of bundles $$\mathrm{T}^{1,0}\T_m\simeq \Psi_m^*\mathcal{H}_A$$ on $\T_m$ only depends on the tangent vector $v$.
Moreover, since $\Psi_m^*\mathcal{H}_A$ extends trivially to $(\Psi_m^H)^*\mathcal{H}_A$ on $\TT$ due to the trivial monodromy along $\ZZ\setminus \Z$,
it gives the identification of the tangent bundle $\mathrm{T}^{1,0}\TT$ with the Hodge subbundle $(\Psi_m^H)^*\mathcal{H}_A$ on $\TT$ by Lemma \ref{Hodge extension}.

From this special feature of the period map, one can see that the identification of the tangent bundle of the $\T_{m}$ with the Hodge subbundle in the definition of strong local Torelli is compatible with both the affine structures on $\T_m$ and $A$, as well as the affine map $\Psi_m$.

Now we recall a lemma due to Griffiths and Wolf, which is proved as Corollary 2 in \cite{GW}.
\begin{lemma}\label{covering-lemma}
Let $f:\, X\to Y$ be a local diffeomorphism of connected Riemannian manifolds. Assume that $X$ is complete for the induced metric. Then $f(X)=Y$, $f$ is a covering map and $Y$ is complete.
\end{lemma}

Since the Hodge metric on $\T^H_m$ is induced by $\Psi_m^H$ and is complete, we have the following corollary from Lemma \ref{covering-lemma}.
\begin{corollary}\label{cover-TmH}
The holomorphic map $\Psi_m^H:\, \TT\to A\cap D$ is a universal covering map, and the image $\Psi_{m}^{H}(\TT)=A\cap D$ is complete with respect to the Hodge metric.
\end{corollary}

Then by the universal properties of the universal covering
$$\Psi_m^H:\, \TT\to A\cap D,$$ or the uniqueness of the universal
cover, we have the following corollary.

\begin{corollary}\label{TmH}
For any $m_1, m_2\ge 3$, the extended Teichm\"uller spaces $\T_{m_1}^H$ and $\T_{m_2}^H$ are  biholomorphic to each other.
\end{corollary}

Since Corollary \ref{TmH} implies that $\TT$ does not depend on $m$, we can simply denote each $\TT$ by $\T^{H}$, rewrite $i_m:\, \T\to \T^H_m$ as $i_\T:\, \T \to \T^H$, and denote the corresponding period maps and affine maps by $$\Phi^{H} :\, \T^{H}\to D\text{ and }\Psi^{H} :\, \T^{H}\to A\cap D$$ respectively. Moreover for each $\Z$, one has the following commutative diagram.
\begin{equation}\label{main digram}
\xymatrix{\T \ar[r]^{i_{\T}}\ar[d]^{\pi_m}&\T^{H}\ar[d]^{\pi_m^H}\ar[r]^{\Phi^H}&D\ar[d]^{\pi_D}\\
\Z\ar[r]^{i}&\ZZ\ar[r]^{\Phi_{\ZZ}}&D/\Gamma.
}
\end{equation}

Now we prove that the universal covering map $\Psi^H:\, \T^{H}\to A\cap D$ is in fact injective,  hence we have $$\T^{H}\simeq A\cap D,$$ and $A\cap D$ is simply connected.

We will give two proofs of this result.
The first proof is to show directly that $A\cap D$ is simply connected. The second proof uses the affine structures on $\T^H$ and $A\cap D$ in a more substantial way. Both the proofs reflect different geometric structures of the period domain and period map which will be useful for further study, so we include both the proofs in this section.

For the first proof, we introduce some notations following \cite{LS}.
Let $K\subset G_{\mathbb{R}}$ be the maximal compact subgroup, whose Lie algebra $\mathfrak{k}_{0}$ is
$$\mathfrak{k}_{0}=(\oplus_{k \text{ even }}\mathfrak{g}^{k,-k})\cap \mathfrak{g}_{0}.$$
In fact, we define $\mathfrak{k}=\oplus_{k \text{ even }}\mathfrak{g}^{k,-k}$, then $\mathfrak{k}_{0}=\mathfrak{k}\cap \mathfrak{g}_0$. Similarly, we define
$$\mathfrak{p}=\oplus_{k \text{ odd }}\mathfrak{g}^{k,-k}\text{ and }\mathfrak{p}_0=\mathfrak{p}\cap \mathfrak{g}_0.$$ Then we have the decompositions
\begin{align*}
\mathfrak{g}=\mathfrak{k}\oplus\mathfrak{p}, \quad \mathfrak{g}_0=\mathfrak{k}_0\oplus \mathfrak{p}_0.
\end{align*}
Let $$\mathfrak{p}_{+}=\mathfrak{p}\cap \mathfrak{n}_{+}=\oplus_{k\ge 1,k \text{ odd }}\mathfrak{g}^{-k,k}.$$ Then  $\mathfrak{p}_+$ can be viewed as an Euclidean subspace of
$\mathfrak{n}_+$ with the induced metric from $\mathfrak{n}_{+}$. Similarly $\mathrm{exp}(\mathfrak{p}_+)$ can be viewed as an Euclidean subspace of $N_{+}$ with the induced
metric from $N_{+}$.

From Corollary 3.2 in \cite{LS}, one gets that the natural projection $$\pi:\,D\to G_\mathbb{R}/K,$$ when restricted to the underlying real manifold of
$\mathrm{exp}(\mathfrak{p}_+) \cap D$, is given by the diffeomorphism
\begin{equation}\label{pi+}
\pi_{+}:\, \mathrm{exp}(\mathfrak{p}_+) \cap D\longrightarrow \mathrm{exp}(\mathfrak{p}_0) \stackrel{\simeq}{\longrightarrow} G_\mathbb{R}/K.
\end{equation}

Indeed, let $\bar{o}$ be the base point in $\mathrm{exp}(\mathfrak{p}_+) \cap D$ as in Lemma 3.1 of \cite{LS}. Given  $Y\in \mathfrak{p}_+$, let  $$\exp (Y) \bar{o}\in
\mathrm{exp}(\mathfrak{p}_+) \cap D$$ denote the left translation of the base point $\bar{o}$ by $\text{exp}(Y)$ in $\mathrm{exp}(\mathfrak{p}_+) \cap D$. Similarly, let $X\in
\mathfrak{p}_0$ with the relation that
$$X =T_0( Y +\tau_0(Y))$$ where $\tau_{0}$ is the complex conjugate of $\mathfrak{g}$ with respect to the real form $\mathfrak{g}_{0}$, for some real number $T_{0}$
determined uniquely in Lemma 3.1 of \cite{LS} from the Harish-Chandra argument, and $$\exp (X) \bar{o}\in \mathrm{exp}(\mathfrak{p}_0) \simeq G_\mathbb{R}/K$$ denote the left
translation by $\exp(X)$ of the base point $\bar{o}$
 in $G_\mathbb{R}/K$. Then we have
$$\pi_+(  \exp (Y) \bar{o}) = \exp (X) \bar{o}.$$

\begin{thm}\label{iso-TmH}
Suppose that the polarized manifold $(X,L)$ belongs to T-class and the strong local Torelli holds for $(X,L)$.
Then the holomorphic covering map $\Psi^H:\, \mathcal{T}^H\to A\cap D$ is an injection and hence a biholomorphic map.
\end{thm}
\begin{proof}
Consider the diffeomorphism $\pi_+:\, \text{exp}(\mathfrak{p}_+) \cap D\to G_{\mathbb R}/K$ discussed above.
By Griffiths transversality, one has
$$\mathfrak{a}\subset\mathfrak{g}^{-1,1}\subset \mathfrak{p}_{+}\text{ and }\mathfrak{a}_0=\mathfrak{a} +\tau_0(\mathfrak{a})\subset \mathfrak{p}_0.$$
Then $A\cap D$ is a submanifold of $\text{exp}(\mathfrak{p}_+) \cap D$, and the diffeomorphism $\pi_{+}$ maps $A\cap D\subseteq \text{exp}(\mathfrak{p}_{+}) \cap D$ diffeomorphically to its image $\exp(\mathfrak{a}_0)$ inside $G_{\mathbb R}/K$, from which one has the diffeomorphism $$A\cap D\simeq \text{exp} (\mathfrak{a}_0)$$ induced by $\pi_+$. Since $\text{exp} (\mathfrak{a}_0)$ is simply connected, one concludes that $A \cap D$ is also simply connected.

Since $\T^H$ is simply connected and $\Psi^H:\, \T^H\to A\cap D$ is a covering map, we conclude that $\Psi^H$ must be a biholomorphic map.
\end{proof}

For the second proof, we will first prove the following elementary lemma, in which we mainly use the completeness of the Hodge metric on $\T^H$, the holomorphic affine structure on $\mathcal{T}^H$, the affineness of $\Psi^H$, and the properties of Hodge metric. We remark that as $\mathcal{T}^H$ is a complex affine manifold, we have the notion of straight lines in it with respect to the affine structure.

\begin{lemma}\label{straightline} For any point in $\mathcal{T}^H,$ there is a straight line segment in $\mathcal{T}^H$ connecting it to the base point $p$.
\end{lemma}
\begin{proof} The proof uses crucially the facts that the map $\Psi^H:\, \T^H\to A\cap D$ is an local isometry with the Hodge metrics on $\T^H$ and $A\cap D$, and that it is also
an affine map with the induced affine structure on $\T^H$.

Let $q$ be any point in $\T^H$. As $\T^H$ is connected and complete with the Hodge metric, by Hopf-Rinow theorem, there exists a geodesic $\gamma$ in $\T^H$ connecting the base point $p$ to $q$. Since
$\Psi^{H}$ is a local isometry with the Hodge metric, $\tilde{\gamma}=\Psi^{H}(\gamma)$ is also a geodesic in $A\cap D$.

By Griffiths transversality, we have that $\mathfrak{a}\subset\mathfrak{g}^{-1,1}\subset \mathfrak{p}_{+}$ and $\mathfrak{a}_0=\mathfrak{a} +\tau_0(\mathfrak{a})\subset \mathfrak{p}_0$.
Restricting the diffeomorphism \eqref{pi+} to $A\cap D$, we know that any geodesic starting from the base point $\tilde{p}=\Psi^{H}(p)$
in $A\cap D$ is of the form $\exp(tX)\tilde{p}$ with $X\in \mathfrak{a}_{0}$ and $t\in \mathbb{R}$.

Recall that, from the computation in the proof of Lemma 3.1 in \cite{LS} which is due to Harish-Chandra in proving  his famous embedding theorem of Hermitian symmetric spaces
as bounded domains in complex Euclidean spaces, we have the relation
$$\exp(tX)\tilde{p}=\exp(T(t)Y)\tilde{p}$$ with $Y\in \mathfrak{a}$ satisfying $X=T_0(Y+\tau_{0}(Y))$ for some  $T_0\in \mathbb{R}$, and $T(t)$ a smooth real valued monotone function of $t$.

 Note that the affine structure on $A\cap D$ is induced from the affine structure on $\mathfrak{a}$ by the exponential map $$\text{exp}:\, \mathfrak{a}\to A,$$
therefore the geodesic $$\tilde{\gamma}=\exp(T(t)Y)\tilde{p}$$ corresponds to a straight line with respect to the affine structure on $A\cap D$. Hence $\gamma$ is also a
straight line in $\T^H$ with respect to the induced affine structure, since $\Psi^H$ is an affine map.
\end{proof}


\begin{proof}[Second Proof of Theorem \ref{iso-TmH}]
Let $q_{1}, q_{2}\in \mathcal{T}^H$ be any two different points. Suppose that to the contrary, one has $\Psi^{H}(q_{1})=\Psi^{H}(q_{2})$, we will
derive a contradiction.

First, suppose that one of $q_{1}, q_{2}$, say $q_{1}$, is the base point $p$, then Lemma \ref{straightline} implies that there is a straight line segment
$l\subseteq \mathcal{T}^H$ connecting $p$ and $q_{2}$. Since $\Psi^{H}$ is an affine map and is locally biholomorphic by local Torelli theorem, the straight line segment $l$ is mapped to a straight line segment $\tilde{l}=\Psi^{H}(l)$ in $$A\cap D \subset A\simeq \C^{N}.$$ But
$\Psi^{H}(p)=\Psi^{H}(q_{2})$, which implies that the straight line segment $\tilde{l}$ is also a cycle in $A\simeq \C^{N}$, which is a contradiction.

Now suppose that both $q_{1}$ and $q_{2}$ are different from the base point $p$, then by Lemma \ref{straightline}, there exist two different straight line segments
$l_{1},l_{2}\subseteq \mathcal{T}^H$ connecting $p$ to $q_{1}$ and $q_{2}$ respectively. Since $$\Psi^{H}(q_{1})=\Psi^{H}(q_{2}),$$ the two straight
line segments $\tilde{l_{i}}=\Psi^{H}(l_{i})$, $i=1,2$ in $A\cap D$ intersect at two different points, $\Psi^H(p)$ and $\Psi^H(q_1)=\Psi^H(q_2)$,
therefore must coincide. This contradicts to the fact that $\Psi^{H}$ is locally biholomorphic at the base point $p$, and the assumption that the two straight
line segments $l_1$ and $l_2$ are different.
\end{proof}

Since $\Psi^H=P\circ \Phi^H$ and the corresponding tangent map
$\text{d}\Psi^H=\text{d}P\circ \text{d}\Phi^H$, we have the
following corollary from Theorem \ref{iso-TmH}.
\begin{corollary}\label{injectivity of P^H}
Let the conditions be as Theorem \ref{iso-TmH}.
Then the extended period map $\Phi^H:\,\T^{H}\to N_+\cap D$ is an embedding.
\end{corollary}
Now we go on to the Torelli space $\T'$. Let $\pi'_{m}:\, \T'\to \Z$ be the covering map for $m\ge 3$. Then there exists the natural covering map $\pi:\, \T \to \T'$ such that the following diagram commutes
$$\xymatrix{
\T \ar[dr]^-{\pi} \ar[dd]^-{\pi_m} &\\
&\T'\ar[dl]^-{\pi_m'},\\
\Z &}$$
which together with diagram \eqref{main digram} gives the following commutative diagram:
\begin{equation}\label{main digram2}
\xymatrix{
\T \ar[dr]^-{\pi}\ar[rr]^-{i_{\T}} \ar[dd]^-{\pi_m} &&\T^{H} \ar[dd]^-{\pi_{m}^{H}}\ar[rr]^-{\Phi^{H}}  &&D\ar[dd]^-{\pi_{D}}\\
&\T'\ar[dl]^-{\pi'_{m}}\ar[urrr]^-{\Phi'}&&&\\
\Z \ar[rr]^-{i} &&\ZZ \ar[rr]^-{\Phi_{\ZZ}} &&D/\Gamma .}
\end{equation}

Let $\T_0= i_\T(\T)$ denote the image of $i_\T$. Next we will define a map $$\pi_{0}:\, \T_{0}\to \T'$$ such that the following commutative diagram holds,
\begin{equation}\label{main diagram2}
\xymatrix{
\T \ar[dr]^-{\pi}\ar[rr]^-{i_{\T}} \ar[dd]^-{\pi_m} &&\T_{0} \ar[dl]_-{\pi_{0}} \ar[dd]^-{\pi_{m}^{H}|_{\T_{0}}}\ar[rr]^-{\Phi^{H}|_{\T_{0}}}  &&D\ar[dd]^-{\pi_{D}}\\
&\T'\ar[dl]^-{\pi'_{m}}\ar[urrr]^-{\Phi'}&&&\\
\Z \ar[rr]^-{i} &&\ZZ \ar[rr]^-{\Phi_{\ZZ}} &&D/\Gamma .}
\end{equation}

The proof of the following proposition is given by explicitly constructing $\pi_0$. Again the level structures plays a crucial role in the argument.

\begin{proposition}\label{im-injective}
Let $\T_0= i_\T(\T)$ be defined by the image of $i_\T$. Then
there is a biholomorphic  map $\pi_0: \, \T_0\to \T'$ such that diagram (\ref{main diagram2}) is commutative.
\end{proposition}

\begin{proof}  We start from the definition of $\pi_0$. For any point $o$ in $\T_{0}$, we can choose $p\in \T$ such that $i_{\T}(p)=o$. We define $$\pi_0:\, \T_{0}\to \T'$$ such that $\pi_0(o) = \pi (p)$. Then clearly $\pi_0$ fits in the above commutative diagram.

We first show that the map $\pi_0:\, \T_0\to \T'$ defined this way is well-defined and satisfies $$\Phi^{H}|_{\T_{0}}=\Phi'\circ \pi_{0}.$$

Let $p\neq q\in \T$ be two points such that
$i_{\T}(p)=i_{\T}(q)\in \T_{0}$.
We choose some $m\ge 3$ such that $\T^{H}\simeq \TT$. Then $$i_m(p)=i_m(q)\in \T_{m} \subset \TT.$$

Let $[X_p, L_p, \gamma_p]$ and $[X_q, L_q, \gamma_q]$ denote the fibers over the points $p$ and $q$ of the analytic family $\U'\to \T'$ respectively, where $\gamma_p$ and $\gamma_q$ are two markings identifying the fixed lattice $\Lambda$ isometrically with $H^{n}(X_{p},\mathbb{Z})/\text{Tor}$ and $H^{n}(X_{q},\mathbb{Z})/\text{Tor}$ respectively.

Since $i_m(p)=i_m(q)$ and $i\circ \pi_m=\pi^H_{m}\circ i_m$, we have $$i\circ\pi_m(p)=i\circ\pi_m(q)\in \Z,$$ i.e. $$\pi_m(p)=\pi_m(q)\in \Z.$$ By the definition of $\Z$, there
exists a biholomorphic map $f:\, X_p\to X_q$ such that $f^*L_q=L_p$ and
$$f^*\gamma_q= \gamma_p\cdot A,$$
where $A\in \text{Aut}(H^{n}(X_{p},\mathbb{Z})/\text{Tor},Q))$ satisfies
$$A=(A_{ij})\equiv\text{Id}\quad(\text{mod } m), \text{ for }m\ge 3.$$

Let $m_0$ be an integer such that $m_0> |A_{ij}|$ for any $i,j$. Corollary \ref{TmH} implies that $$\T^{H}\simeq \TT\simeq \T_{m_0}^{H}.$$  The same argument as above implies
that $$\pi_{m_0}(p)=\pi_{m_0}(q)\in \mathcal{Z}_{m_{0}}$$ and hence
$$A=(A_{ij})\equiv\text{Id}\quad(\text{mod } m_0).$$
Since each $m_0> |A_{ij}|$, we have $A=\text{Id}$.

Therefore, we have found a biholomorphic map $$f:\, X_p\to X_q$$ such that $f^*L_q=L_p$ and $f^*\gamma_q= \gamma_p$. This implies that $p$ and $q$ in $\T$ actually correspond to the same point in the Torelli space $\T'$, i.e.
$$\pi(p) =\pi(q) \text{ in } \T'.$$
So we have proved that the map $\pi_0:\, \T_0\to \T'$ is well-defined.

From the definitions of the period map and that of the Torelli space, we deduce that the map $\pi_0$ satisfies the identity $$\Phi^{H}|_{\T_{0}}=\Phi'\circ \pi_{0}.$$ Clearly
$\pi_0$ is surjective, because $\pi:\, \T \to \T'$ is surjective. Since $\Phi^{H}:\, \T^{H}\to D$ is injective, so is the restriction map $$\Phi^{H}|_{\T_{0}}:\, \T_{0}\to D,$$
which implies the injectivity of the map $\pi_{0}:\, \T_{0}\to \T'$. Hence $\pi_{0}:\, \T_{0}\to \T'$ is in fact a biholomorphic map. This completes the proof of the
proposition.
\end{proof}

Define the injective map $$\pi^{0}:\, \T' \to \T^{H}, \  \pi^{0}=(\pi_{0})^{-1},$$ we then have the relation that $$\Phi'=\Phi^{H}\circ \pi^{0}:\, \T' \to D.$$ From the
injectivity of $\Phi^{H}$ and $\pi^{0}$, we deduce the global Torelli theorem on the Torelli space as follows.
\begin{thm}[Global Torelli theorem]\label{Global Torelli theorem}
Suppose that the polarized manifold $(X,L)$ belongs to the T-class, and the strong local Torelli holds for $(X,L)$, then the global Torelli theorem holds on the Torelli space $\T'$, i.e., the period map $\Phi':\, \T'\to D$ is injective.
\end{thm}





Let $$\Psi' = P\circ \Phi':\, \T'\to A\cap D$$ be the composite of $\Phi'$ with the projection map $P:\, N_+\cap D \longrightarrow A\cap D$, we then have,
\begin{corollary}
Under the same conditions in Theorem \ref{Global Torelli theorem}, the holomorphic map $\Psi':\, \T'\to A\cap D$ is injective.
\end{corollary}

Since the moduli spaces with level $m$ structure of polarized K3 surfaces, Calabi-Yau manifolds and hyperk\"ahler manifolds are smooth for $m$ large, and they have the strong local Torelli property as described in Section \ref{local Torelli},  we have the following corollary.

\begin{corollary}\label{case 1}
Let the polarized manifold $(X,L)$ be one of the following cases,
\begin{enumerate}
\item[(i)] K3 surface;
\item[(ii)] Calabi-Yau manifold;
\item[(iii)] hyperk\"ahler manifold.
\end{enumerate}
Then the global Torelli theorem holds on the Torelli space $\T'$ for $(X,L)$.

\end{corollary}

See introduction and references there for discussions about previous results related to Corollary \ref{case 1}.

By Theorem 1.2 in \cite{JL1} and the result of Benoist as stated in Proposition 2.2 of \cite{JL}, we know that the moduli spaces of smooth hypersurfaces have smooth covers given by the moduli spaces with level structures, and for the complete intersections used in the hyperplane arrangements, the corresponding moduli spaces also have smooth covers.
So when $m$ is large, the moduli space  $\Z$ containing $(X,L)$ is  smooth as required in the definition of T-class. As discussed in Section \ref{local Torelli}, some of these examples have strong local Torelli property, therefore we have the following corollary.

\begin{corollary}\label{case 3}

Let the polarized manifold $(X,L)$ be one of the following cases,
\begin{enumerate}
\item[(i)] smooth hypersurface of degree $d$ in $\mathbb{P}^{n+1}$ satisfying $d|(n+2)$ and $d\ge 3$;
\item[(ii)] the polarized manifold of $\mathbb{P}^{m-1}$ associated to an arrangement of $m$ hyperplanes in $\mathbb{P}^{n}$ with $m\ge n$, defined as in \eqref{variety of hyperplane} of Section \ref{eigenperiods};
\item[(iii)] smooth cubic surface or cubic threefold.
\end{enumerate}
Then global Torelli theorem holds on the corresponding Torelli spaces $\T'$.

\end{corollary}

We remark that case (i) of Corollary \ref{case 3} is new, except the work of Voisin in \cite{Voisin99} which proves the generic Torelli theorem for the moduli space of quintic threefolds.

Case (ii) of Corollary \ref{case 3} with $n=1$ can be considered as a version on the Torelli space of the main result of \cite{DM}, which is the famous Deligne-Mostow theory. Corollary \ref{case 3} in case (ii) with $n>1$ is new.

As mentioned in the introduction, in \cite{ACT02} and \cite{ACT11}, Allcock, Carlson and Toledo proved global Torelli theorems on the moduli space and Torelli space of smooth cubic surfaces or cubic threefolds. Case (iii) of Corollary \ref{case 3} can be considered as a version of their results in \cite{ACT02} and \cite{ACT11} on the Torelli spaces.

Applying Proposition \ref{im-injective}, we have that $\T^H$, which is biholomorphic to $A\cap D$, is actually the Hodge metric completion of $\T'$ with respect to the induced Hodge metric. Moreover we have the following theorem.

\begin{thm}\label{pseudoconvex}
Suppose that the polarized manifold $(X,L)$ belongs to T-class and strong local Torelli holds for $(X,L)$.
Then the Hodge metric completion $\T^{H}\simeq A\cap D$ of $\T'$ is a bounded pseudoconvex domain in $A\simeq \C^{N}$. In particular, there exists a unique complete K\"ahler-Einstein metric on $\T^{H}\simeq A\cap D$ with Ricci curvature $-1$.
\end{thm}

First we recall some notions from complex analysis in several variables. One can refer to \cite{Hom} for details.

An open and connected subset $\Omega \subseteq \C^{N}$ is called a domain in $\C^{N}$.
A function $u$ defined on a domain $\Omega \subseteq \C^{N}$ with values in $[-\infty,+\infty]$ is called plurisubharmonic if
\begin{itemize}
\item[(1)] $u$ is semicontinous from above;
\item[(2)] For any $z,w \in \C^{N}$, the function $t \mapsto u(z+tw)$ is subharmonic at the points of $\C$ where it is defined.
\end{itemize}

A domain $\Omega \subseteq \C^{N}$ is called pseudoconvex if there exists a continous plurisubharmonic function $u$ on $\Omega$ such that for any $c\in \mathbb{R}$, the set
$$\Omega_{c}=\{z:\,z\in \Omega, u(z)<c \}$$
is a relatively compact subset of $\Omega$. Such function $u$ is called an exhaustion function on $\Omega$.

From Theorem 2.6.2 in \cite{Hom}, we know that a $C^2$ function $u$ on a domain $\Omega \subseteq\C^N$ is plurisubharmornic if and only if its Levi form is positive definite at any point in $\Omega$.

In their paper \cite{GS}, Griffiths and Schmid proved the following proposition.

\begin{proposition}\label{pc of D}
On every manifold $D$, which is dual to a K\"ahler C-space, there exists a $C^{\infty}$ exhaustion function $f$, whose Levi form, restricted to $T^{1, 0}_h(D)$, is positive definite at every point of $D$.
\end{proposition}

\begin{proof}[Proof of Theorem \ref{pseudoconvex}]
Boundedness of $A\cap D$ comes from the proof of Lemma 3.1 in \cite{LS}.

Let $f$ be the $C^{\infty}$ function on the period domain $D$ as in Proposition \ref{pc of D}.
Then the restricted function $g=f|_{A\cap D}$ has positive definite Levi form at any point of $A\cap D$, since the tangent space at any point of $A\cap D$ lies in the fiber of $T^{1, 0}_h(D)$.

To complete the proof, we only need to show that $g$ is an exhaustion function on $A\cap D$.
In fact, for any $c\in \mathbb{R}$, the set
$$\Omega_{c}=\{z\in A\cap D:\, g(z)\le c\}=\{z\in D:\, f(z)\le c\}\cap (A\cap D)$$
is compact in $A\cap D$, since $$\{z\in D:\, f(z)\le c\}$$ is compact in $D$, see the proof of Theorem 8.1 in \cite{GS}, and  $A\cap D$ with the induced Hodge metric is
complete, so is a complete subset of $D$, and hence closed in $D$.

The existence and uniqueness of the complete K\"ahler-Einstein metric on $\T^{H}\simeq A\cap D$ with Ricci curvature $-1$ follows directly from the results of \cite{MokYau}.
\end{proof}



\section{Applications}\label{general}
In this section, we still consider polarized manifolds in the T-class for which the strong local Torelli holds.
We will use the results in Section \ref{global Torelli} to prove that, in this case, the global Torelli theorem holds on the moduli space with level $m$ structure.

Suppose that the polarized manifold $(X,L)$ belongs to T-class and strong local Torelli holds for $(X,L)$. Let $\Z$ containing $(X,L)$
be defined in Definition \ref{T-class} and $\ZZ$, $\T$, $\T'$  and $\TT$ be defined as before.

Let $\Gamma=\rho(\pi_1(\ZZ))$ denote the global monodromy group,  where $$\rho:\, \pi_1(\ZZ) \to \text{Aut}(H_{\mathbb Z}, Q)$$ denotes the monodromy representation. Then the
corresponding period maps can be written in the following commutative diagram,
\[\xymatrix{\mathcal{T}^{H}\ar[rd]^{\tilde{\Phi}^{H}}\ar[d]^{\pi_{m}^{H}}\ar[r]^-{\Phi^{H}}&D\ar[d]^{\pi_D}\\ \ZZ\ar[r]^{\Phi_{\ZZ}}&D/\Gamma,}\]
where $$\tilde{\Phi}^{H}=\pi_D\circ\Phi^{H}.$$

The image of the extended period map $\Phi_{\ZZ}$ in general is an analytic subvariety of $D/\Gamma$.  We refer the reader to page 156 of \cite{Griffiths3} for details of the
analyticity of the image of the period map.

\begin{theorem}\label{generic}
Suppose that the polarized manifold $(X,L)$ belongs to T-class and strong local Torelli holds for $(X,L)$.
Then the extended period map $\Phi_{\ZZ}: \,\ZZ\rightarrow D/\Gamma$ is injective. As a consequence the global Torelli theorem holds on the moduli space $\Z$ with level $m$ structure.
\end{theorem}

\begin{proof}[Proof of Theorem \ref{generic}]
We will give two proofs. For the first proof, we first show that $\Phi_{\ZZ}$ is a covering map from $\ZZ$ to its image in $D/\Gamma$. This follows from the following lemma.

\begin{lemma}\label{lemma smoothness}
Let $\tilde{\Phi}^{H}:\,{\mathcal{T}^{H}}\to D/\Gamma$ be the composition of
$\Phi^{H}$ and the covering map $\pi_D:\,D\to D/\Gamma$. Then
$\tilde{\Phi}^H$ is a covering map onto its image  which is $\Phi_{\ZZ}(\ZZ)$.
\end{lemma}

\begin{proof}[Proof of Lemma \ref{lemma smoothness}]
First from Theorem (D.2) in page 179 of \cite{Griffiths3}, $\Phi^H(\T^H)$ is invariant under the action of $\Gamma$. In our case, $\Phi^H$ is a global embedding which implies that $\Phi^H(\T^H)$ is smooth, and $\pi_D$ is a covering map.

On the other hand $D/\Gamma$ is smooth, because $\Gamma$ is torsion-free from Lemma \ref{trivial monodromy}.  As discussed in page 156 of \cite{Griffiths3}, the isotropy group corresponding to the points in the smooth manifold $\Phi^H(\T^H)$ that are mapped to a singular point in  $\Phi_{\ZZ}(\ZZ)$ by the quotient of $\Gamma$ must be a finite torsion subgroup of $\Gamma$. Therefore $$\Phi_{\ZZ}(\ZZ)=\pi_D\circ \Phi^H(\T^H)$$ is actually a smooth variety, since $\Gamma$ is torsion-free. This gives that $\T^H$, which is biholomorphic to $\Phi^H(\T^H)$, is also a covering space of $\Phi_{\ZZ}(\ZZ)$, and $\tilde{\Phi}^H$ is a covering map.
\end{proof}

\begin{lemma}\label{P^H covering}
The extended period map $\Phi_{\ZZ}: \,\ZZ\rightarrow D/\Gamma$ is a covering map from $\ZZ$ onto its image in $D/\Gamma$.
\end{lemma}
\begin{proof}[Proof of Lemma \ref{P^H covering}]
Note that in the following commutative diagram
\[\xymatrix{\mathcal{T}^{H}\ar[rd]^{\tilde{\Phi}^{H}}\ar[d]^{\pi_{m}^{H}}\ar[r]^-{\Phi^{H}}&D\ar[d]^{\pi_D}\\ \ZZ\ar[r]^{\Phi_{\ZZ}}&D/\Gamma,}\]
all the varieties involved are smooth. Since the map $\pi_{m}^{H}$ is a covering map, $\pi_{m}^{H}$ is locally biholomorphic.
Similarly, $\Phi^H$ is an embedding and $\pi_D$ is a covering map.

From Lemma \ref{lemma smoothness}, we know that the map
$$\tilde{\Phi}^H = \pi_D\circ \Phi^H= \Phi_{\ZZ}\circ \pi_{m}^H:\, \T^H \to  \Phi_{\ZZ}(\ZZ)$$ is a covering map, therefore a locally biholomorphic map, and hence the Jacobians of these maps are all nondegenerate.  In particular this implies that the Jacobian of $\Phi_{\ZZ}$ is nondegenerate, which implies that $$\Phi_{\ZZ}:\, \ZZ \to \Phi_{\ZZ}(\ZZ)$$ is a locally biholomorphic map.

On  the smooth complex submanifold $$\Phi_{\ZZ}(\ZZ)=\pi_D\circ \Phi^H(\T^H)$$ of $D/\Gamma$, we have the induced Hodge metric from $D/\Gamma$. Since the image of $\ZZ$ under $\Phi_{\ZZ}$ is precisely
$\Phi_{\ZZ}(\ZZ)$ and the map $\Phi_{\ZZ}$ is nondegenerate, the pull-back Hodge metric on $\ZZ$ by $\Phi_{\ZZ}$ from the submanifold $\Phi_{\ZZ}(\ZZ)$ in $D/\Gamma$ is the same as the original Hodge metric induced from $D/\Gamma$ through the pull-back by $\Phi_{\ZZ}:\, \ZZ \to D/\Gamma$.

Since the induced Hodge metric on $\ZZ$ is complete as the Hodge metric completion of $\Z$, we can directly apply Lemma \ref{covering-lemma} of Griffiths--Wolf to deduce that  $$\Phi_{\ZZ}:\,\ZZ\to \Phi_{\ZZ}(\ZZ)$$ is a covering map.
\end{proof}


\noindent\textit{Proof of Theorem \ref{generic}(continued).}
In Lemma \ref{P^H covering} we have already proved that $\Phi_{\ZZ}(\ZZ)$ is smooth and $$\Phi_{\ZZ}:\,\ZZ\to \Phi_{\ZZ}(\ZZ)$$ is a covering map.
This implies that $\pi_1(\ZZ)$ can be identified as a subgroup of the fundamental group of $\Phi_{\ZZ}(\ZZ)$.

On the other hand,  since $\T^H$ is simply connected and $\Phi^H:\, \T^H \to D$ is an embedding, Lemma \ref{lemma smoothness} gives that
$$ \Phi_{\ZZ}(\ZZ)= \Phi^H(\T^H)/\Gamma$$ which implies that the fundamental group of  $\Phi_{\ZZ}(\ZZ)$ is $\Gamma$, where
$$\Gamma= \rho(\pi_1(\ZZ))\simeq \pi_1(\ZZ)/\text{Ker}(\rho)$$ with $\rho:\, \pi_1(\ZZ)\to \text{Aut}(H_{\mathbb Z}, Q)$  the monodromy representation. From this we deduce that $$\pi_1(\ZZ)\subset \Gamma\simeq \pi_1(\ZZ)/\text{Ker}(\rho),$$
which implies that $\text{Ker}(\rho)=0$ and $\pi_1(\ZZ)\simeq \Gamma$.

Therefore we have proved that $\pi_1(\ZZ)$ is isomorphic to the fundamental group of $\Phi_{\ZZ}(\ZZ)$.
From this we deduce that the covering map
$$\Phi_{\ZZ}: \ZZ \to D/\Gamma$$
is a biholomorphic map onto its image.
\end{proof}

The proof of Theorem \ref{generic} uses Lemma \ref{lemma smoothness} and Lemma \ref{P^H covering}, which is proved under the condition of strong local Torelli.  In fact, we can
also prove Theorem \ref{generic}, once the conclusion of Theorem \ref{Global Torelli theorem} holds.

\begin{proof}[Second proof of Theorem \ref{generic}]
Let $p_{1}$ and $p_{2}$ be two points in $\Z$ such that $\Phi_{\Z}(p_{1})=\Phi_{\Z}(p_{2})$ in $D/\Gamma$. Let $$[X_{i},L_{i},[\gamma_{i}]_{m}], \, i=1,2$$ be the fibers over $p_{1}$ and $p_{2}$ of the analytic family $f_{m} :\, \U_{m}\to \Z$ respectively.  From Proposition \ref{im-injective} we know that $$\pi^0:\, \T'\to \T^{H}$$ identifies $\T'$ to the Zariski open submanifold $$\T_0=i_\T(\T)\simeq i_m(\T)\subseteq \TT$$ of $\TT$, and that $\T_0$ is a cover of $\Z$ by Lemma \ref{openness3}. In the following discussion we will use freely the identification $$\T_0\simeq \T'.$$

There exist two points ${q}_{1}$ and ${q}_{2}$ in $\T'$ over which are the fibers $$[X_{i},L_{i},\gamma_{i}], \, i=1,2$$ of the analytic family $g':\, \U'\to \T'$ respectively,
such that
$$\pi_{m}'({q}_{i})=p_{i}, \, i=1,2$$
under the covering map $\pi_{m}':\, \T' \to \Z$.

The condition $\Phi_{\Z}(p_{1})=\Phi_{\Z}(p_{2})$ implies that there exists $\gamma\in \Gamma$ such that
$$\Phi'({q}_{1})=\gamma\Phi'({q}_{2}).$$
Let ${q}_{1}'\in \T'$ correspond to the triple $[X_{1},L_{1},\gamma_{1}\gamma]$. Then by the definition of the period map $\Phi'$ in \eqref{defn of P'}, we have
\begin{eqnarray*}
\Phi'(q_{1}')&=&(\gamma_{1}\gamma)^{-1}(F^n(X_{1})\subseteq\cdots\subseteq F^0(X_{1}))\\
             &=&\gamma^{-1}\gamma_{1}^{-1}(F^n(X_{1})\subseteq\cdots\subseteq F^0(X_{1}))\\
             &=&\gamma^{-1}\Phi'(q_{1})\\
             &=&\Phi'(q_{2}).
\end{eqnarray*}
By Theorem \ref{Global Torelli theorem}, we see that $q_{1}'=q_{2}$ in $\T'$.

Now we have the inclusion $\pi^{0}:\, \T' \to \T^{H}$ and the inclusion $i:\, \Z \to \ZZ$, therefore we can view the points $p_{1}$ and $p_{2}$ as points in $\ZZ$ and the points $q_{1}$, $q'_{1}$ and $q_{2}$ as points in $\T^{H}$.

Since $\gamma$ lies in the image of the monodromy representation $\rho:\, \pi_{1}(\ZZ)\to \Gamma$, there exists a $\tilde{\gamma}\in \pi_{1}(\ZZ)$ such that
$$\rho(\tilde{\gamma})=\gamma \text{ and } q_{1}'=q_{1}\tilde{\gamma},$$ where $q_{1}\tilde{\gamma}$ is defined by the action of $\pi_{1}(\ZZ)$ on the universal cover $\T^{H}$, and the action of $\tilde{\gamma}$ on the fiber $[X_{1},L_{1},\gamma_{1}]$ over $q_{1}$ is defined by
$$[X_{1},L_{1},\gamma_{1}]\tilde{\gamma}= [X_{1},L_{1},\gamma_{1}\gamma].$$
So we have  $$p_{1}=\pi_{m}^{H}(q_{1})=\pi_{m}^{H}(q_{1}')=\pi_{m}^{H}(q_{2})=p_{2},$$
which proves the injectivity of $\Phi_{\Z}$.
\end{proof}

\vspace{+12 pt}

\noindent Center of Mathematical Sciences, Zhejiang University, Hangzhou, Zhejiang 310027, China;\\
Department of Mathematics, University of California at Los Angeles, Los Angeles, CA 90095-1555, USA\\
\noindent e-mail: liu@math.ucla.edu, kefeng@cms.zju.edu.cn

\vspace{+6pt}
\noindent Center of Mathematical Sciences, Zhejiang University, Hangzhou, Zhejiang 310027, China \\
\noindent e-mail: syliuguang2007@163.com


\begin{thebibliography}{99}

\bibitem{ACT02}
D. Allcock, J. Carlson, and D. Toledo,
\newblock{The complex hyperbolic geometry of the moduli space of cubic surfaces,}
\newblock{\em J. Alg. Geom. }, \textbf{11} (2002), pp.~659-724.

\bibitem{ACT11}
D. Allcock, J. Carlson, and D. Toledo,
\newblock{The Moduli Space of Cubic Threefolds as a Ball Quotient,}
\newblock{\em Mem. Amer. Math. Soc. },  \textbf{209} (2011), xii+70.

\bibitem{Beau}
A. Beauville,
\newblock{Moduli of cubic surfaces and Hodge theory (after Allcock, Carlson, Toledo),}

\newblock{ \em G\'eom\'etriesa courbure n\'egative ou nulle, groupes discrets et rigidit\'es, S\'eminaires et Congres.} Vol. 18. 2009.

\bibitem{BR}
D.~Burns and M.~Rapoport,
\newblock{On the Torelli problem for k\"ahlerian K-3 surfaces,}
\newblock{\em Ann. Sc. ENS. }, \textbf{8} (1975), pp.~235--274.

\bibitem{CKT}
J. ~Carlson, A. ~Kasparian, and D. ~Toledo,
\newblock{ Variations of Hodge structure of maximal dimension,}
\newblock{\em Duke Journal of Math,} \textbf{58} (1989), pp.~669--694.

\bibitem{CMP}

J.~Carlson, S.~Muller-Stach, and C.~Peters,
\newblock{\em Period Mappings and Period Domains},
\newblock{ Cambridge University Press,} {(2003)}.


\bibitem{CZTG}
E. ~Cattani, F. ~El Zein, P. A. ~Griffiths, and L. D. ~Trang,
\newblock{\em Hodge Theory},
\newblock{Mathematical Notes}, \textbf{49},
\newblock{Princeton University Press,} (2014).

\bibitem{Schmid86}
E.~Cattani, A.~Kaplan, and W.~Schmid,
\newblock Degeneration of Hodge structures,
\newblock {\em Ann. of Math. }, \textbf{123} (1986), pp.~457--535.

\bibitem{CLS12}
X.~Chen, K.~Liu and Y.~Shen,
\newblock{Global Torelli theorem for projective manifolds of Calabi--Yau type}, {\tt arXiv:1205.4207v3}.

\bibitem{Debarre}
O.~Debarre,
\newblock{Periods and moduli,}
\newblock{\em Current Developments in Algebraic Geometry,} MSRI Publications, \textbf{59} (2011), pp.~65--84.

\bibitem{DM}
P. ~Deligne and G.W. ~Mostow,
\newblock{Monodromy of hypergeometric functions and non-lattice integral monodromy,}
\newblock{\em Publ. Math. IHES}, \textbf{63} (1972), pp.~543--560.

\bibitem{DK}
I. V. ~Dolgachev and S. ~Kondo,
\newblock{Moduli of K3 Surfaces and Complex Ball Quotients,}
\newblock{\em Arithmetic and Geometry Around Hypergeometric Functions,} Progress in Mathematics,
\textbf{260} (2007), pp.~43--100.


\bibitem{For}
F.~Forstneri$\check{c}$,
\newblock{\em Stein manifolds and holomorphic mappings: the homotopy
principle in complex analysis},
\newblock{Ergebnisse der Mathematik und ihrer Grenzgebiete 3. Folge}, \textbf{56},
\newblock{ Springer-Verlag, Berlin-Heidelberg}, (2011).

\bibitem{Fri}
R.~Friedman,
\newblock{A New Proof of the Global Torelli Theorem for K3 Surfaces,}
\newblock{\em Annals of Mathematics,} \textbf{120} (1984), pp.~237--269.


\bibitem{GR}
H.~Grauert and R.~Remmert,
\newblock{\em Coherent Analytic Sheaves},
\newblock{Grundlehren der mathematischen Wissenschaften}, \textbf{265},
\newblock{Springer-Verlag, Berlin-Heidelberg-NewYork-Tokyo}, (2011).

\bibitem{Green14}
M.~Green, P.~Griffiths, and C.~Robles,
\newblock {Extremal degenerations of polarized Hodge structures,}
\newblock   {\em arXiv:1403.0646}.

\bibitem{Griffiths1}
P.~Griffiths,
\newblock{Periods of integrals on algebraic manifolds I,}
\newblock{\em Amer. J. Math. }, \textbf{90} (1968), pp.~568--626.

\bibitem{Griffiths2}
P.~Griffiths,
\newblock{Periods of integrals on algebraic manifolds II,}
\newblock{\em Amer. J. Math. }, \textbf{90} (1968), pp.~805--865.

\bibitem{Griffiths69}
P.~Griffiths,
\newblock{On the Periods of Certain Rational Integrals: I and II,}
\newblock{\em Annals of Mathematics,} \textbf{90} (1969), pp.~460--495 and 496--541.

\bibitem{Griffiths3}
P.~Griffiths,
\newblock{Periods of integrals on algebraic manifolds III,}
\newblock{\em Publ. Math. IHES. }, \textbf{38} (1970), pp.~125--180.

\bibitem{Griffiths4}
P.~Griffiths,
\newblock{Periods of integrals on algebraic manifolds: Summary of main results and discussion of open problems,}
\newblock{\em Bull. Amer. Math. Soc. }, \textbf{76}, no.{2} (1970), pp.~228--296.

\bibitem{Grif84}
P.~Griffiths,
\newblock {\em Topics in transcendental algebraic geometry},
\newblock {Annals of Mathematics Studies}, Volume \textbf{106}, Princeton University Press, Princeton, NJ, (1984).

\bibitem{GS}
P.~Griffiths and W.~Schmid,
\newblock{Locally homogeneous complex manifolds,}
\newblock{\em Acta Math. }, \textbf{123} (1969), pp.~253--302.

\bibitem{GW}
P.~Griffiths and J.~Wolf,
\newblock{Complete maps and differentiable coverings,}
\newblock{\em Michigan Math. J. }, Volume \textbf{10}, Issue 3 (1963), pp.~253--255.

\bibitem{HC}
Harish-Chandra,
\newblock{Representation of semisimple Lie groups VI,}
\newblock{\em Amer. J. Math.,} \textbf{78} (1956), pp.~564--628.


\bibitem{Hel}
S.~Helgason,
\newblock{\em Differential geometry, Lie groups, and symmetric spaces,}
\newblock{Academic Press, New York, (1978)}.

\bibitem{Hom}
L.~H\"ormander,
\newblock{\em An introduction to complex analysis in several variables},
\newblock{Van Nostrand, Princeton, NJ, (1973)}.

\bibitem{IY}
Y.~Ilyashenko and S.~Yakovenko,
\newblock{\em Lectures on Analytic Theory of Ordinary Differential Equations,}
\newblock{Graduate Studies in Mathematics,} \textbf{86}, Amer. Math. Soc., (2008).

\bibitem{JL}
A.~Javanpeykar and D.~Loughran,
\newblock{Complete Intersections: Moduli, Torelli, and Good Reduction},
\newblock{\em arXiv: 1505.02249}, (2015).

\bibitem{JL1}
A.~Javanpeykar and D.~Loughran,
\newblock{ The Moduli of Hypersurfaces with Level Structure},
\newblock{\em arXiv: 1511.09291}, (2015).



\bibitem{KU}
K.~Kato and S.~Usui.
\newblock {\em Classifying spaces of degenerating polarized Hodge structures},
\newblock {Annals of Mathematics Studies. }, \textbf{169},
Princeton University Press, Princeton, NJ, (2009).


\bibitem{Kobayashi}
S.~Kobayashi,
\newblock {\em Hyperbolic manifolds and holomorphic mappings},
\newblock{ Bull. Amer. Math. Soc. }, \textbf{78}, (1972).


\bibitem{KM}
K.~Kodaira and J.~Morrow,
\newblock {\em Complex manifolds},
\newblock AMS Chelsea Publishing, Porvindence, RI, (2006), Reprint of the 1971 edition with errata.

\bibitem{Konno2}
K.~Konno,
\newblock {Infinitesimal Torelli theorem for complete intersections in certain homogeneous K\"ahler manifolds, II,}
\newblock {\em Tohoku Math. J. }, \textbf{42}, no.3 (1990), pp.~333--338.

\bibitem{Konno91}
K.~Konno,
\newblock {On the variational Torelli problem for complete intersections,}
\newblock {\em Compositio Mathematica,} \textbf{78} (1991), pp.~271--296.

\bibitem{Lee}
J. Lee,
\newblock{\em Introduction to Smooth Manifolds},
\newblock{Graduate Texts in Mathematics}, \textbf{218},
\newblock{Springer-Verlag, Berlin-Heidelberg-NewYork-Tokyo}, 2nd ed., (2013).

\bibitem{LS}
K. ~Liu and Y.~ Shen,
\newblock{Boundedness of the images of period maps and applications},
\newblock{\em arXiv:1507.01860}, (2015).

\bibitem{CLS13}
K.~Liu and Y.~Shen,
\newblock{Hodge metric completion of the Teichm\"uller space of Calabi-Yau manifolds}, {\tt arXiv:1305.0231. }

\bibitem{LS1}
K. ~Liu, Y.~ Shen, and X.~ Chen,
\newblock{Applications of the affine structures on the Teichm\"uller spaces},
\newblock{\em Geometry and Topology of Manifolds,} 10th China-Japan Conference 2014, Futaki, A., Miyaoka, R., Tang, Z., Zhang, W. (Eds.).


\bibitem{Looijenga07}
E. ~Looijenga,
\newblock{Uniformization by Lauricella Functions: An Overview of the Theory of Deligne-Mostow},
\newblock{\em Arithmetic and Geometry Around Hypergeometric Functions Progress in Mathematics}, \textbf{260} (2007), pp.~207--244.

\bibitem{Looijenga}
E. ~Looijenga,
\newblock{Trento notes on Hodge theory,}
\newblock{\em Rend. Sem. Mat. Univ. Politec. Torino}, \textbf{69}, no.2 (2011), pp.~113--148.

\bibitem{LooPet}
E.~Looijenga, C.~Peters,
\newblock{Torelli Theorems for K\"ahler K3 Surfaces,}
\newblock{\em Compositio Mathematica}, \textbf{42} (1981), pp.~145--186.

\bibitem{Mayer}
R. ~Mayer,
\newblock{Coupled contact systems and rigidity of maximal variations of Hodge
structure,}
\newblock{\em Trans. AMS,} \textbf{352}, no.5 (2000), pp.~2121--2144.


\bibitem{Mckay}
B.~Mckay,
\newblock{Extension phenomena for holomorphic geometric structures},
\newblock {\em Symmetry, Integrability and Geometry: Methods and Applications}, \textbf{5} (2009), 058, 45 pages.

\bibitem{Milne}

J. ~Milne,
\newblock{Shimura varieties and moduli,}
\newblock{\em Handbook of Moduli}, \textbf{2} (2011), pp.~467--548.

\bibitem{MY}
Y.~Mitera and J.~Yoshizaki,
\newblock{The local analytical triviality of a complex analytic singular foliation,}
\newblock{\em Hokkaido Math. J. }, \textbf{33}, no.{2} (2004), pp.~275--297.

\bibitem{Mok}
N.~Mok,
\newblock{\em Metric Rigidity Theorems on Hermitian Locally Symmetric Manifolds,}
\newblock{Ser. Pure Math. }, Vol. \textbf{6}, World Scientific, Singapore-New Jersey-London-HongKong, (1989).

\bibitem{MokYau}
N.~Mok and S.-T.~Yau,
\newblock{Completeness of the K\"ahler-Einstein metric on bounded domains and the characterization of domain of holomorphy by curvature condition,}
\newblock{\em The mathematical heritage of Henri Poincar\'e, Part 1,}
\newblock{\em Sympo. in Pure Math. }, \textbf{39} (1983), pp.~41--60.

\bibitem{Mukai}
S. ~Mukai, and W. M. ~Oxbury,
\newblock{\em An Introduction to Invariants and Moduli,}
\newblock{Cambridge Studies in Advanced Mathematics,} Vol. \textbf{80}, Cambridge University Press, (2003).



\bibitem{Peters1}
C.~Peters,
\newblock {The local Torelli theorem I, Complete Intersections,}
\newblock {\em Mathematische Annalen}, \textbf{217}, Issue 1 (1975), pp.~1--16,.

\bibitem{Peters}
C.~Peters,
\newblock {On the local torelli theorem, a review of known results,}
\newblock {\em Variétés Analytiques Compactes,}
\newblock {\em Lecture Notes in Mathematics,} \textbf{683} (1978), pp.~62--73.

\bibitem{PSS}
I.~Pjateckii-Sapiro and I. Shafarevich,
\newblock {Torelli's theorem for algebraic surfaces of type K3,}
\newblock {\em Math. USSR Izv. }, \textbf{5} (1971), pp. ~547--588.

\bibitem{Popp3}
H.~Popp,
\newblock {On moduli of algebraic varieties III, Fine moduli spaces,}
\newblock {\em Compositio Mathematica}, \textbf{31}, Issue {3} (1975), pp.~ 237--258.

\bibitem{Popp}
H. ~Popp,
\newblock {\em Moduli theory and classification theory of algebraic varieties},
\newblock {Lecture Notes in Mathemathics}, \textbf{620},
Springer-Verlag, Berlin-New York, (1977).

\bibitem{Saito}
M. H. ~Saito,
\newblock {Weak global Torelli theorem for certain weighted projective hypersurfaces,}
\newblock {\em Duke Math. J. ,} \textbf{53} (1986), pp.~67--111.

\bibitem {schmid1}
W.~Schmid,
\newblock {Variation of {H}odge structure: the singularities of the period mapping},
\newblock {\em Invent. Math. }, \textbf{22} (1973), pp.~211--319.

\bibitem {Schwartz}
M. H. ~Schwartz,
\newblock {Lectures on stratification of complex analytic sets,}
\newblock {\em Tata Institute of Fundamental Research,} \textbf{38}, (1966).

\bibitem{Serre60}
J.-P. ~Serre,
\newblock {Rigidit\'{e} du foncteur de Jacobi d’\'{e}chelon n≥3},
\newblock{\em S\'{e}m. Henri Cartan}, \textbf{13}, no.17 (1960/61), Appendix.

\bibitem{Serre}
J.-P. ~Serre,
\newblock {\em Complex semisimple Lie algebras},
\newblock {Springer Monographs in Mathematics}, Springer-Verlag, (2001).


\bibitem{SU}
Y.~Shimizu and K.~Ueno,
\newblock {\em Advances in moduli theory},
\newblock {Translation of Mathematical Monographs}, \textbf{206}, American Mathematics Society, Providence, Rhode Island, (2002).

\bibitem{Sommese}
A. ~ Sommese,
\newblock{On the rationality of the period mapping},
\newblock{\em Annali della Scuola Normale Superiore di Pisa-Classe di Scienze}, \textbf{5}, Issue 4 (1978), pp.~683--717.

\bibitem{Sugi}
M.~Sugiura,
\newblock{Conjugate classes of Cartan subalgebra in real semi-simple Lie algebras,}
\newblock{\em J. Math. Soc. Japan,} \textbf{11} (1959), pp.~374--434.

\bibitem{Sugi1}
M.~Sugiura,
\newblock{Correction to my paper: Conjugate classes of Cartan subalgebra in real semi-simple Lie algebras},
\newblock{\em J. Math. Soc. Japan,} \textbf{23} (1971), pp.~374--383.

\bibitem{sz}
B. Szendr\"oi,
\newblock{Some finiteness results for Calabi-Yau threefolds},
\newblock{\em J. London Math. Soc,} (2), \textbf{60} (1999), pp.~689--699.

\bibitem{Varchenko}
A. ~Varchenko,
\newblock{Hodge filtration of hypergeometric integrals associated with an affine connfiguration of hyperplanes and a local Torelli theorem},
\newblock{\em I. M. Gelfand Seminar, Adv. Soviet Math.}, \textbf{16} (1993), pp. ~167--177.

\bibitem{Verbitsky}
M.~Verbitsky,
\newblock{Mapping class group and a global Torelli theorem for hyperk\"ahler manifolds},
\newblock{\em Duke Math. J.},  \textbf{162} (2013), pp.~2929-2986.

\bibitem{Viehweg}
E.~Viehweg,
\newblock {\em Quasi-projective Moduli for Polarized Manifolds,}
\newblock {Ergebnisse der Mathematik und ihrer Grenzgebiete 3. Folge}, \textbf{30}, Springer-Verlag, (1995).

\bibitem{Voisin99}
C.~Voisin,
\newblock{A generic Torelli theorem for the quintic threefold,}
\newblock {\em New trends in algebraic geometry,}
\newblock{London Math. Soc. Lecture note Series,} \textbf{264} (1999), pp.~425-464.

\bibitem{Voisin}
C.~Voisin,
\newblock {\em Hodge theory and complex algebraic geometry I},
\newblock Cambridge Universigy Press, New York, (2002).

\bibitem{Xu}
Y.~Xu,
\newblock{\em Lie groups and Hermitian symmetric spaces},
\newblock{Science Press in China}, (2001). (Chinese)

\end{thebibliography}
\end{document}